\documentclass[10pt, a4paper, twocolumn ,titlepage, egregdoesnotlikesansseriftitles]{scrartcl}
\pdfoutput=1
\usepackage{geometry}
\usepackage{mathtools} % required for \coloneqq
\usepackage{stmaryrd} % required for \lightning
\usepackage{amsmath} % required for \numberwithin
\usepackage{amsthm} % required for \newtheorem
\usepackage{amssymb} % required for \preccurlyeq
\usepackage[english]{babel} % required for hyphenation
\usepackage[hidelinks]{hyperref}
\usepackage[all]{xy} % required for \xymatrix for Commutative Diagrams
\usepackage{graphicx} 
\usepackage{chngpage} % required for adjustwidth, as used in abstract
\usepackage{indentfirst} % required to indent first lines of paragraphs
\usepackage[T1]{fontenc}
\usepackage[latin9]{inputenc}
\usepackage{prettyref} % required for \prettyref
\usepackage[all]{xy} % required for \xymatrix for Commutative Diagrams
\usepackage{xcolor}
\usepackage[font={small,it}]{caption} % required for smaller captions
\usepackage{sectsty} % required for smaller section headings

\setkomafont{section}{\large}
\graphicspath{{figures/}}

\numberwithin{equation}{section}
\numberwithin{figure}{section}

\theoremstyle{plain}
\newtheorem{thm}{Theorem}
\newtheorem{lem}[thm]{Lemma}
\newtheorem{prop}[thm]{Proposition}
\newtheorem{cor}[thm]{Corollary}

\theoremstyle{definition}
\newtheorem*{defn*}{Definition}
\newtheorem*{example*}{Example}

\newcommand{\bigslant}[2]{{\raisebox{.2em}{$#1$}\left/\raisebox{-.2em}{$#2$}\right.}}
\newcommand{\orcid}[1]{\href{https://orcid.org/#1}{\includegraphics[width=10pt]{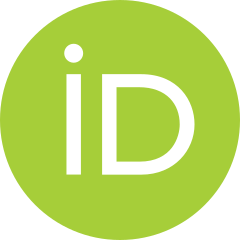}  orcid.org/#1}}

\title{\vspace{-0.5cm}\LARGE \textsc{\scalebox{0.92}[1.0]{Foundations of Structural Statistics:}\\\scalebox{0.92}[1.0]{Topological Statistical Theory}}\vspace{-0.5cm}}
\author{\normalsize P. Michl \orcid{0000-0002-6398-0654}}
\date{}

\begin{document}

\twocolumn[
\begin{@twocolumnfalse}  

\maketitle

\begin{abstract}
\vspace{-1.9cm}
\begin{adjustwidth}{10mm}{10mm}
Topological statistical theory provides the foundation for a modern mathematical reformulation of classical statistical theory: Structural Statistics emphasizes the structural assumptions that accompany distribution families and the set of structure preserving transformations between them, given by their statistical morphisms. The resulting language is designed to integrate complicated structured model spaces like deep-learning models and to close the gap to topology and differential geometry. To preserve the compatibility to classical statistics the language comprises corresponding concepts for standard information criteria like sufficiency and completeness.
\\\\
\textbf {Keywords:} Statistical Equivalence, Statistical Morphism, Topological Statistical Model
\end{adjustwidth}
\vspace{0.5cm}
\end{abstract}
\end{@twocolumnfalse}
] 

\section{Introduction}
As a passionate advocate of observation based statistical inference in the first third of the $20^{\text{th}}$ century \textsc{R. A. Fisher} introduced the concepts of \emph{sufficient statistics}, as well as a measure of information, which is invariant with respect to them \cite{Fisher1935}. This \emph{Fisher information} encouraged \textsc{S. Kullback} and \textsc{R. A. Leibler} in the mid-$20^{\text{th}}$ century to further investigations, regarding the correspondence between statistics and the uprising information theory of \textsc{C. E. Shannon }\cite{Kullback1959}. In the course of their analysis a further invariant structure became apparent, which became known as the \emph{Kullback-Leibler divergence} \cite{Kullback1951}. Although this structure permitted to quantify invariance properties, it did not provide a generic framework, which motivated the incorporation of reversible Markov processes and further statistical \emph{divergence functions} \cite{Csiszar1963,Morimoto1963}. A generic theory of equivalence and invariance of statistical models, however, still is missing! A first step in this direction requires to substantiate the objects of desire:

\begin{defn*}[Statistical Model]
\label{def:Statistical-model}\emph{Let $(S,\,\Sigma)$ be a measurable space - the sample space, $T\colon(\Omega,\,\mathcal{F},\,\mathrm{P})\to(S,\,\Sigma)$ an i.i.d. random variable - a statistic, and $\mathcal{M}$ a set of probability distributions over $(S,\,\Sigma)$  - the model space. Then the 3-tuple $(S,\,\Sigma,\,\mathcal{M})$ is termed a statistical model.}
\end{defn*}

\section{\label{subsec:Statistical-Equivalence}From Statistical Equivalence to Statistical Morphisms}
The choice of a statistical model is generally subjected to the convenience of the experimenter and therefore far away from being consistent! This ambiguity concerns the underlying sample space and statistic, by the applied physical measurement method, as well as the model space, by the interpretation of the structural knowledge. As observation based statistical inference however is obligated to provide unique conclusions at the end of the day, an immediate equivalence relationship is induced by statistical models that allow the same conclusions for the same observations.

The \emph{Observational Equivalence} of statistical models describes the circumstance, that no outcome of the respectively underlying statistical population may be used to distinguish which model provides a closer estimation of the population distribution and consequentially of any population parameter. As observational equivalence therefore precisely describes the preservation of observation based statistical inference, it may be regarded and emphasized as a canonical blueprint for \emph{Statistical Equivalence}. The disadvantage of this description, however, is it's Brutforce approach, which makes it inapplicable for complicated statistical models.

A quite different approach exploits the representation of the model spaces in their embedding function space: Very intuitively the model spaces may be regarded as sets of particles within their surroundings and statistical equivalence is provided, if and only if the respective particle sets can reversibly be transferred to each other by \emph{Markov processes}. Thereby the reversibility assures, that the orbits of observational distinguishable distribution assumptions do not intersect, such that ``no information is lost'' while the Markov property assures, that the target model space is completely determined by the domain model space, such that ``no additional information is generated''. This  encourages the following definition:
\begin{defn*}[Statistical Equivalence]
\label{def:Statistical-equivalence}\emph{Let $(X,\,\mathcal{A},\,\mathcal{M})$ and $(Y,\,\mathcal{B},\,\mathcal{N})$ be statistical models. Then they are termed statistically equivalent, iff there exists a reversible Markov process from $X$ to $Y$, which pushes $\mathcal{M}$ to $\mathcal{N}$.}
\end{defn*}
Although the above definition stands out for its ease of comprehension and simplicity, it loses the coupling between the model spaces, the sample spaces and the underlying statistics. Therefore in the following a novel approach is developed, that closes the gaps of the above given definitions. Let

\[
T_{X}\colon(\Omega,\,\mathcal{F},\,\mathrm{P})\to(X,\,\mathcal{A})
\]
and

\[
T_{Y}\colon(\Omega,\,\mathcal{F},\,\mathrm{P})\to(Y,\,\mathcal{B})
\]
be the underlying statistics, then 
\[
Z\coloneqq\{T_{X},\,T_{Y}\}
\]
is a stochastic processes and trivially satisfies the Markov property, such that for any $P\in\mathcal{M}$ and $Q\in\mathcal{N}$ there exists an a.e. unique transition probability $\mathrm{p}_{X,Y}$. Conversely also 
\[
Z^{-1}\coloneqq\{T{}_{Y},\,T_{X}\}
\]
satisfies the Markov property, such that for $P$ and $Q$ there also exists an a.e. unique transition probability $\mathrm{p}_{Y,X}$. With respect to a given pair $P$ and $Q$ the reversibility of $Z$ may therefore be verified by the principle of \emph{detailed balance}, which postulates that $Z$ is reversible, iff an \emph{equilibrium distribution} $\pi$ over $(X\times Y,\,\mathcal{A}\otimes\mathcal{B})$ exists, such that 
\[
\mathrm{p}_{12}(y\mid x)\pi(x)=\mathrm{p}_{21}(x\mid y)\pi(y),\,\forall(x,\,y)\in X\times Y
\]
This criterion however appears to be trivially satisfied by \emph{Bayes' theorem}: 
\[
\mathrm{P}(A\mid B)P(B)=\mathrm{P}(B\mid A)P(A),\,\forall(A,\,B)\in\mathcal{A}\otimes\mathcal{B}
\]
However in order to establish this relationship the conditional dependencies $\mathrm{P}(A\mid B)$ have to be endowed with a density function. As the extension of conditional probabilities to continuous probability spaces, however, provides some peculiarities, further considerations have to be taken into account: Let $(X,\,\mathcal{A},\,\mu)$ and $(Y,\,\mathcal{B},\,\nu)$ be continuous probability spaces and by abuse of notation let ``$\mathrm{P}(y\mid x)$'' denote the probability of an outcome $y\in Y$ under the precondition of $x\in X$, regardless of whether it's well-defined. Then a naive interpretation of the term$\mathrm{P}(y\mid x)$ for any given $y\in Y$ on the one hand implicates, that $\mathrm{P}(y\mid x)=0$ for all $x\in X$, since:

\[
\mathrm{P}(y\mid x)=\mu(Z_{1}\circ Z_{2}^{-1}(y)\cap x)\leq\mu(x)=0
\]
On the other hand however, there has at least to be one $x\in X$ with $\mathrm{P}(y\mid x)>0$, since:

$$\int_{X}\mathrm{P}(y\mid x)\mathrm{d}\mu(x)=1>0\,\lightning$$

In order to avoid an inconsistent notation the conditional probabilities are therefore represented by \emph{transition operators} and afterwards generalized to a continuous notation.
\begin{lem}
\label{lem:1} Let $(X,\,\mathcal{A},\,\mu)$ be a discrete probability space with $\mathcal{A}=\mathcal{P}(X)$ and $(Y,\,\mathcal{B})$ a discrete measurable space with $\mathcal{B}=\mathcal{P}(Y)$. Let further be 
\[
T\colon(X,\,\mathcal{A})\to(Y,\,\mathcal{B})
\]
a measurable function, which induces the transition operator \emph{$\tau$}
by $\tau(\mu)=T_{*}\mu$ and $\mathrm{P}(y\mid x)$ the conditional
probabilities from $X$ to$Y$, then: 
\begin{equation}
\mathrm{P}(y\mid x)=\tau(\delta_{x})(y),\,\forall(x,\,y)\in X\times Y\label{eq:lem:3.4:1}
\end{equation}
\end{lem}

\begin{proof}
For $x\in X$ let $\delta_{x}$ denote the Dirac measure over $(X,\,\mathcal{A})$, then:
\begin{equation}
\tau(\delta_{x})(T(x))=\sum_{a\in X}\tau(\delta_{x})(T(a))\delta_{x}(\{a\})\label{eq:lem:3.4:2}
\end{equation}
Furthermore by applying the law of total probability, it also holds that:
\begin{equation}
\tau(\delta_{x})(T(x))=\sum_{a\in X}\mathrm{P}(T(x)\mid a)\delta_{x}(\{a\})\label{eq:lem:3.4:3}
\end{equation}
A subsequent substitution of equation \ref{eq:lem:3.4:3} in equation \ref{eq:lem:3.4:2} yields:

\begin{equation}
\sum_{a\in X}\mathrm{P}(T(x)\mid a)\delta_{x}(\{a\})=\sum_{a\in X}\tau(\delta_{x})(T(a))\delta_{x}(\{a\})\label{eq:lem:3.4:4}
\end{equation}
And since $\delta_{x}(\{a\})=1\Leftrightarrow x=a$ equation \ref{eq:lem:3.4:4} dissolves to a representation:
\[
\mathrm{P}(T(x)\mid x)=\tau(\delta_{x})(T(x))
\]
This proves lemma \ref{lem:1} for $y=T(x)$. Since $\mathrm{P}(y\mid x)=0$
for $y\notin T(x)$ and 
\[
\mathrm{P}(y_{1}\cup y_{2}\mid x)=\mathrm{P}(y_{1}\mid x)\mathrm{P}(y_{2}\mid x)
\]
 for $y_{1}\cap y_{2}=\emptyset$ the generic case follows by complete
induction over the discrete product measurable space $(X\times Y,\,\mathcal{A}\otimes\mathcal{B})$.
\end{proof}
\begin{defn*}[Regular conditional probability]
\label{def:Regular-conditional-probabilty} \emph{Let $(X,\,\mathcal{A},\,\mu)$ be a probability space, $(Y,\,\mathcal{B})$ a measurable space and $T\colon(X,\,\mathcal{A})\to(Y,\,\mathcal{B})$ a measurable function,
which induces the transition operator $\tau$ by $\tau(\mu)=T_{*}\mu$. Then the regular conditional probability from $X$ to $Y$ is given by:
\begin{equation}
\mathrm{p}(y\mid x)\coloneqq\frac{\mathrm{d}\tau(\delta_{x})}{\mathrm{d}\tau(\mu)}(y),\,\forall(x,\,y)\in X\times Y\label{eq:def:Regular-conditional-probability:1}
\end{equation}
}
\end{defn*}
Due to the definition regular conditional probabilities resolve the contradiction of infinitesimal measures by the properties of transition operators and therefore essentially by the Radon-Nikodym derivative. Although the definition only regards the case of a conditional probability $\mathrm{p}(y\mid x)$ from $X$ to $Y$ it implicates a dual conditional probability $\mathrm{p}(x\mid y)$ from $Y$ to $X$, which existence and uniqueness is postulated by the following lemma.

\begin{lem}[\emph{Dual conditional probability}]
\label{lem:2}Let $(X,\,\mathcal{A},\,\mu)$ be a probability space, $(Y,\,\mathcal{B})$ a measurable space and 
\[
T\colon(X,\,\mathcal{A})\to(Y,\,\mathcal{B})
\]
a measurable function\emph{,} which induces the transition operator $\tau$ by $\tau(\mu)=T_{*}\mu$. Then the regular conditional probability from $Y$ to $X$ is a.e. uniquely given by: 
\begin{equation}
\mathrm{p}(x\mid y)\stackrel{\text{a.e.}}{=}\frac{\mu(x)}{\tau(\mu)(y)}\frac{\mathrm{d}\tau(\delta_{x})}{\mathrm{d}\tau(\mu)}(y),\,\forall(x,\,y)\in X\times Y\label{eq:lem:3.5:1}
\end{equation}
\end{lem}

\begin{proof}
Let $\nu\coloneqq T_{*}\mu$, then lemma \ref{lem:1} postulates the
existence of a transition operator $\tau$ that satisfies $\tau(\mu)=\nu$.
By applying the definition of the regular conditional probability
$\mathrm{p}(y\mid x)$, a probability measure $\pi$ over the product
measurable space $(X\times Y,\,\mathcal{A}\otimes\mathcal{B})$ is
given by:
\begin{equation}
\pi(x,\,y)=\frac{\mathrm{d}\tau(\delta_{x})}{\mathrm{d}\nu}(y)\mu(x)\label{eq:lem:3.5:2}
\end{equation}
Conversely lemma \ref{lem:1} also postulates the existence of a further
transition operator $\tau^{*}$, that satisfies $\tau^{*}(\nu)=\mu$.
Then $\tau^{*}$ induces a probability measure $\pi^{*}$ over the
dual product measurable space $(Y\times X,\,\mathcal{B}\otimes\mathcal{A})$
by:
\begin{equation}
\pi^{*}(y,\,x)=\frac{\mathrm{d}\tau^{*}(\delta_{y})}{\mathrm{d}\mu}(x)\nu(y)\label{eq:lem:3.5:3}
\end{equation}
Let $(\pi^{*})^{*}$ be the dual product measure of $\pi^{*}$, then
$\pi=(\pi^{*})^{*}$ for $\pi$-almost all $(x,\,y)\in X\times Y$.
Let now be $\mathrm{p}(x\mid y)$ be the regular conditional probability,
which is given by $\tau^{*}$, then the substitution of lemma
\ref{lem:2} and equation \ref{eq:lem:3.5:2} yields a representation:
\[
\mathrm{p}(x\mid y)=\frac{\mathrm{d}\tau^{*}(\delta_{y})}{\mathrm{d}\mu}(x)=\frac{\mu(x)}{\nu(y)}\frac{\mathrm{d}\tau(\delta_{x})}{\mathrm{d}\nu}(y)
\]
The property of $\tau$ to be unique for $\mu$-almost all $x\in X$
and of $\tau^{*}$ to be unique for $\nu$-almost all $y\in Y$ prove
that $\mathrm{p}(x\mid y)$ is a.e. unique. Equation \ref{eq:lem:3.5:1}
therefore follows by the identity$\nu=\tau(\mu)$.
\end{proof}
\begin{thm}[\emph{Continuous Bayes' Theorem}]
\label{thm:Continuous-Bayes-Theorem}Let $(X,\,\mathcal{A},\,\mu)$
and $(Y,\,\mathcal{B},\,\nu)$ be probability spaces with $\mu\preccurlyeq\nu$
and $T\colon(X,\,\mathcal{A})\to(Y,\,\mathcal{B})$ a measurable function
with $\nu=T_{*}\mu$. Then for $x,\,y\in X\times Y$:
\begin{equation}
\mathrm{p}(x\mid y)\stackrel{\text{a.e.}}{=}\frac{\mu(x)}{\nu(y)}\mathrm{p}(y\mid x)\label{eq:lem:3.5:1-1}
\end{equation}
\end{thm}

\begin{proof}
By applying lemma \ref{lem:1} and lemma \ref{lem:2} it immediate
follows that for $\nu$-almost all $(x,\,y)\in X\times Y$ it holds
that:
\[
\mathrm{p}(x\mid y)\stackrel{\text{lem} \ref{lem:2}}{=}\frac{\mu(x)}{\tau(\mu)(y)}\frac{\mathrm{d}}{\mathrm{d}\tau(\mu)}\tau(\delta_{x})(y)\stackrel{\text{lem} \ref{lem:1}}{=}\frac{\mu(x)}{\nu(y)}\mathrm{p}(y\mid x)
\]
\end{proof}
Let now $(X,\,\mathcal{A},\,\{\mu\})$ and $(Y,\,\mathcal{B},\,\{\nu\})$
be singleton statistical models of $(\Omega,\,\mathcal{F},\,\mathrm{P})$
and $Z=\{T_{1},\,T_{2}\}$ a Markov process which is given by the
underlying statistics 
\[
T_{X}\colon(\Omega,\,\mathcal{F},\,\mathrm{P})\to(X,\,\mathcal{A})
\]
and 
\[
T_{Y}\colon(\Omega,\,\mathcal{F},\,\mathrm{P})\to(Y,\,\mathcal{B})
\]
Then the measurable function $T=T_{Y}\circ T_{X}^{-1}$ allows an
application of Bayes' theorem which in turn assures the existence
of an equilibrium distribution $\pi=\mu\nu$ and therefore $Z$ to
be reversible. However for non-singleton statistical models $(X,\,\mathcal{A},\,\mathcal{M})$
and $(Y,\,\mathcal{B},\,\mathcal{N})$ the transition operator, which
is induced by $T$ depends on the choice of $P\in\mathcal{M}$, by
the condition: 
\[
\tau_{P}(P)=T_{*}P
\]
Consequentially also the conditional probability from $X$ to $Y$
which is induced by $P$ depends on this choice, such that:

\begin{equation}
\mathrm{p}_{P}(y\mid x)\coloneqq\frac{\mathrm{d}\tau_{P}(\delta_{x})}{\mathrm{d}\tau_{P}(P)}(y),\,\forall(x,\,y)\in X\times Y\label{eq:conditional-probability}
\end{equation}
As the Markov process, however is completely determined by the underlying
conditional probabilities, the existence of a Markov process, that
induces $\mathcal{N}$ from $\mathcal{M}$ requires $\mathrm{p}_{P}(y\mid x)$
and therefore $\tau_{P}$ to be independent of the choice of $P$.
Then there has to exist a transition operator $\tau$ with $\tau(\mathcal{M})=\mathcal{N}$
that induces a fixes conditional probability $\mathrm{p}(y\mid x)$.
The dual conditional probability from $Y$ to $X$, however still
depends on the choice of $P$, by the continuous Bayes Theorem: 
\[
\mathrm{p}_{P}(x\mid y)\coloneqq\frac{P(x)}{\tau(P)(y)}\frac{\mathrm{d}\tau(\delta_{x})}{\mathrm{d}\tau(P)}(y),\,\forall(x,\,y)\in X\times Y
\]
Then the reversibility of the Markov process is precisely satisfied,
if also the dual conditional probability does not depend on the choice
of $P$. In this case also the equilibrium distribution 
\[
\pi_{P}(x,\,y)=P(x)\tau(P)(y)
\]
 does not depend on the choice of $P$ and therefore allows to generalize
the principle of detailed balance to model spaces.
\begin{prop}[\emph{Generalized principle of detailed balance}]
Let $Z$ be a Markov process from $X$ to $Y$, which pushes $\mathcal{M}$
to $\mathcal{N}$. Then $Z$ is reversible, iff there exist conditional
probabilities $\mathrm{p}(y\mid x)$ and $\mathrm{p}(x\mid y)$, such
that for all $P\in\mathcal{M}$ and its respectively induced $Q\in\mathcal{N}$,
it holds that:
\[
\mathrm{p}(x\mid y)Q(y)=\mathrm{p}(y\mid x)P(x),\,\forall(x,\,y)\in X\times Y
\]
\end{prop}

The generalized principle of detailed balance can by substantiated
if the respective model spaces $\mathcal{M}$ and $\mathcal{N}$ are
dominated by measures $\mu$ and $\nu$. Then the conditional probabilities
are induced by a transition operator 
\[
\tau:L^{1}(X,\,\mathcal{A},\,\mu)\to L^{1}(X,\,\mathcal{A},\,\nu)
\]
 that satisfies $\tau(\mathcal{M})=\mathcal{N}$. This property facilitates
the notation of \emph{Statistical Morphisms}.
\begin{defn*}[Statistical Morphism]
\label{def:Statistical-morphism}\emph{Let $(X,\,\mathcal{A},\,\mathcal{M})$
and $(Y,\,\mathcal{B},\,\mathcal{N})$ be statistical models and 
\[
T\colon(X,\,\mathcal{A})\to(Y,\,\mathcal{B})
\]
 a measurable function. Then a mapping $f\colon\mathcal{M}\to\mathcal{N}$
is termed a Statistical Morphism, iff there exists a transition operator
\[
\tau:L^{1}(X,\,\mathcal{A},\,\mathcal{M})\rightarrow L^{1}(Y,\,\mathcal{B},\,\mathcal{N})
\]
such that the following diagram commutes:}
\[
\xymatrix{\mathcal{M}\ar@{.>}@/^{1pc}/[rrr]^{f}\ar@{_{(}->}@/^{1pc}/[dr]^{\iota} &  &  & \mathcal{N}\ar@{^{(}->}@/_{1pc}/[dl]_{\iota}\\
 & L^{1}(X,\,\mathcal{A},\,\mathcal{M})\ar@{->}[r]^{\tau} & L^{1}(Y,\,\mathcal{B},\,\mathcal{N})\\
 & (X,\,\mathcal{A})\ar@{->}[r]^{T}\ar@{->}[u]\ar@{.>}@/^{1pc}/[luu] & (Y,\,\mathcal{B})\ar@{->}[u]\ar@{.>}@/_{1pc}/[ruu]
}
\]
\end{defn*}
\begin{lem}[\emph{Existence and Uniqueness}]
\label{lem:5}Let $(X,\,\mathcal{A},\,\mathcal{M})$ and $(Y,\,\mathcal{B},\,\mathcal{N})$
be statistical models and 
\[
T\colon(X,\,\mathcal{A})\to(Y,\,\mathcal{B})
\]
a measurable function with $\mathcal{N}=T_{*}\mathcal{M}$. Then $T$
induces a Statistical Morphism $f\colon\mathcal{M}\to\mathcal{N}$
which is a.e. unique up to the quotient $\mu/f(\mu)$.
\end{lem}

\begin{proof}
Lemma \ref{lem:2} postulates the existence of a transition operator
$\tau$ with $\tau(\mu)=T_{*}\mu$, which is unique for $\nu$-almost
all $y\in Y$. Furthermore from $\mathcal{M}\preccurlyeq\mu$ it follows,
that \emph{
\[
\mathcal{M}\subseteq L^{1}(X,\,\mathcal{A},\,\mu)
\]
}and from $\mathcal{N}\preccurlyeq\nu$ that $T_{*}\mu\preccurlyeq\nu$.
Since however $\mathcal{N}=T_{*}\mathcal{M}$ it also follows, that:
\[
\mathcal{N}\subseteq L^{1}(Y,\,\mathcal{B},\,T_{*}\mu)
\]
Then for any $P\in\mathcal{M}$ it holds, that $\tau(P)\in\mathcal{N}$,
and therefore that $\mathrm{img}\tau\mid_{\mathcal{M}}\subseteq\mathcal{N}$.
Conversely for any $Q\in\mathcal{N}$ there exists an $P\in\mathcal{M}$
with $Q=\tau(P)$ such that also $\mathcal{N}\subseteq\mathrm{img}\tau\mid_{\mathcal{M}}$.
This proves that $\mathcal{N}=\mathrm{img}\tau\mid_{\mathcal{M}}$
and thereupon that $f\coloneqq\tau\mid_{\mathcal{M}}$ yields a Statistical
Morphism which is unique for $\nu$-almost all $y\in Y$. As $f$
in turn depends on the choice of $\mu$ it unique up to the quotient
$\mu/f(\mu)$ and therefore essentially a.e. unique. 
\end{proof}
Since transition operators are bounded linear operators with regard
to the underlying $L^{1}$-spaces, they induce a proximity structure
within the model spaces, which is determined by its image and its
kernel. Thereby the image of the transition operator is uniquely given
by: 
\[
\mathrm{img}(\tau)=L^{1}(Y,\,\mathcal{B},\,T_{*}\mu)
\]
and its kernel by $\mathrm{ker}(\tau)=\tau^{-1}(0_{\mathcal{B}})$,
where: 
\[
0_{\mathcal{B}}\coloneqq\{\nu\in\mathrm{img}(\tau)\mid\nu(B)=0,\,\forall B\in\mathcal{B}\}
\]
With respect to a to a Statistical Morphism $f$ however, the situation
is slightly more sophisticated, as the model spaces do not provide
a canonical vector space structure. Due to the requirement of commutativity
with respect to a measurable function $T$ however the image is given
by $\mathrm{img}(f)=T_{*}\mathcal{M}$. Furthermore the kernel of
$f$ corresponds to a partition of $\mathcal{M}$, which is induced
by the subspace topology from $L^{1}$. Thereby $P,\,Q\in\mathcal{M}$
are \emph{$L^{1}$-identical }over $\mathcal{A}$ and denoted by $P\stackrel{\mathcal{A}}{=}Q$
or by the equivalence class $P\in\mathrm{id}_{\mathcal{A}}(Q)$, iff:
\[
P(A)=Q(A),\,\forall A\in\mathcal{A}
\]
Furthermore for any $Q\in\mathrm{img}(f)$ the kernel of $f$ at $Q$
is given by the preimage 
\[
\mathrm{ker}_{Q}(f)=f^{-1}(\mathrm{id}_{\mathcal{B}}(Q))
\]

\begin{defn*}[Statistical Epi-/Mono-/Isomorphism]
\label{def:Statistical-morphisms}\emph{Let $f\colon\mathcal{M}\to\mathcal{N}$
be a Statistical Morphism which is induced by: 
\[
T\colon(X,\,\mathcal{A})\to(Y,\,\mathcal{B})
\]
 Then a $f$ is termed a Statistical Monomorphism, iff it satisfies:
\[
\mathrm{ker}_{f(P)}(f)\subseteq\mathrm{id}_{\mathcal{A}}(P),\,P\in\mathcal{M}
\]
 In this case it follows, that:
\[
f(P)\stackrel{\mathcal{B}}{=}f(Q)\Rightarrow P\stackrel{\mathcal{A}}{=}Q,\,\forall P\in\mathcal{M},\,\forall Q\in\mathcal{M}
\]
Furthermore $f$ is termed a Statistical Epimorphism iff for all $Q\in\mathcal{N}$,
there exists an $P\in\mathcal{M}$, such that $f(P)\stackrel{\mathcal{B}}{=}Q$
and $\mathrm{img}(f)\stackrel{\mathcal{B}}{=}\mathcal{N}$. Finally
$f$ is termed a Statistical Isomorphism, iff it is a Statistical
Monomorphism and a Statistical Epimorphism.}
\end{defn*}
\begin{lem}
\label{lem:4.4-2}Let $(X,\,\mathcal{A},\,\mathcal{M})$ and $(Y,\,\mathcal{B},\,\mathcal{N})$
be statistical models and $f\colon\mathcal{M}\to\mathcal{N}$ be a
Statistical Morphism, then $f$ is a Statistical Isomorphism, iff
there exists a further Statistical Morphism $f^{*}\colon\mathcal{N}\to\mathcal{M}$,
such that (i) 
\[
(f^{*}\circ f)(P)\stackrel{\mathcal{A}}{=}P,\,\forall P\in\mathcal{M}
\]
and (ii) 
\[
(f\circ f^{*})(Q)\stackrel{\mathcal{B}}{=}Q,\,\forall Q\in\mathcal{N}
\]
\end{lem}

\begin{proof}
``$\Longrightarrow$'' Let $f$ be a statistical isomorphism, then
a further Statistical Morphism $f^{*}\colon\mathcal{N}\to\mathcal{M}$
is given by $f^{*}(Q)\coloneqq\mathrm{ker}_{Q}(f)$, which provides
a complete system of representatives of the kernel, such that 
\[
(f^{*}\circ f)(P)\stackrel{\mathcal{A}}{=}P,\,\forall P\in\mathcal{M}
\]
 and 
\[
(f\circ f^{*})(Q)\stackrel{\mathcal{B}}{=}Q,\,\forall Q\in\mathcal{N}
\]
``$\Longleftarrow$'' Conversely if there exists any Statistical
Morphism $f^{*}\colon\mathcal{N}\to\mathcal{M}$ such that $f$ and
$f^{*}$ satisfy (i) and (ii), then for any $P,\,Q\in\mathcal{M}$
with $P\notin\mathrm{id}_{\mathcal{A}}(Q)$ it follows that:
\[
f(P)\notin\mathrm{id}_{\mathcal{B}}(f(Q))
\]
Otherwise let $P\notin\mathrm{id}_{\mathcal{A}}(Q)$ with $f(P)\in\mathrm{id}_{\mathcal{B}}(f(Q))$,
then it would follow, that also 
\[
(f^{*}\circ f)(P)\stackrel{\mathcal{A}}{=}(f^{*}\circ f)(Q)
\]
However since $(f^{*}\circ f)(P)\stackrel{\mathcal{A}}{=}P$ and $(f^{*}\circ f)(Q)\stackrel{\mathcal{A}}{=}Q$
it would also follow, that $P\notin\mathrm{id}_{\mathcal{A}}(Q)$
which contradicts to the initial assumption $P\notin\mathrm{id}_{\mathcal{A}}(Q)$
$\lightning$. This proves, that $f$ is a statistical monomorphism.
Furthermore let $Q\in\mathcal{N}$ and $P=f^{*}(Q)$. Then $f(P)=(f\circ f^{*})(Q)$
and therefore $f(P)\stackrel{\mathcal{B}}{=}Q$. This proves, that
$f$ is a Statistical Epimorphism and therefore a statistical isomorphism.
\end{proof}
\begin{prop}
\label{prop:7}Let $(X,\,\mathcal{A},\,\mathcal{M})$ and $(Y,\,\mathcal{B},\,\mathcal{N})$
be statistical models with $\mathcal{M}\preccurlyeq\mu$ and $\mathcal{N}\preccurlyeq\nu$,
then $(X,\,\mathcal{A},\,\mathcal{M})$ and $(Y,\,\mathcal{B},\,\mathcal{N})$
are statistical equivalent, iff there exists a statistical isomorphism
$f\colon\mathcal{M}\to\mathcal{N}$.
\end{prop}

\begin{proof}
``$\Longrightarrow$'' Let $(X,\,\mathcal{A},\,\mathcal{M})$ and
$(Y,\,\mathcal{B},\,\mathcal{N})$ be statistical equivalent, then
there exists a Markov Process $Z$ from $(X,\,\mathcal{A})$ to $(Y,\,\mathcal{B})$,
that induces $\mathcal{N}$ from $\mathcal{M}$ and therefore a measurable
function 
\[
T\colon(X,\,\mathcal{A})\to(Y,\,\mathcal{B})
\]
 with $\mathcal{N}=T_{*}\mathcal{M}$. Then due to lemma \ref{lem:5}
$T$ induces a Statistical Morphism $f\colon\mathcal{M}\to\mathcal{N}$
and therefore a transition operator 
\[
\tau\colon L^{1}(X,\,\mathcal{A},\,\mu)\to L^{1}(Y,\,\mathcal{B},\,\nu)
\]
 with $\tau(\mathcal{M})=\mathcal{N}$ such that $f\coloneqq\tau\mid_{\mathcal{M}}$.
Since $Z$ is reversible also $\nu$ induces a Statistical Morphism
$f^{*}\colon\mathcal{N}\to\mathcal{M}$ and therefore a dual transition
operator 
\[
\tau^{*}\colon L^{1}(Y,\,\mathcal{B},\,\nu)\to L^{1}(X,\,\mathcal{A},\,\mu)
\]
 with $\tau^{*}(\mathcal{N})=\mathcal{M}$, such that $f^{*}\coloneqq\tau^{*}\mid_{\mathcal{N}}$.
Thereby $f$ and $f^{*}$ are unique up to the quotients $\mu/f(\mu)$
and $\nu/f^{*}(\nu)$. Furthermore $f$ and $f^{*}$have to satisfy
the generalized principle of detailed balance. Therefore independent
of $\mu/f(\mu)$ and $\nu/f^{*}(\nu)$ it follows, that: 
\[
(f^{*}\circ f)(P)\stackrel{\mathcal{A}}{=}P,\,\forall P\in\mathcal{M}
\]
and conversely that:
\[
(f\circ f^{*})(Q)\stackrel{\mathcal{B}}{=}Q,\,\forall Q\in\mathcal{N}
\]
Due to lemma \ref{lem:4.4-2} it follows, that $f$ is a statistical
isomorphism. ``$\Longleftarrow$'' Let $f$ be a statistical isomorphism.
Then there exists a transition operator $\tau$ with $\tau(\mathcal{M})=\mathcal{N}$,
such that the conditional probabilities given by equation \ref{eq:conditional-probability}
do not depend on the choice of $P\in\mathcal{M}$. Then $\tau$ defines
a Markov process $Z$ from $(X,\,\mathcal{A})$ to $(Y,\,\mathcal{B})$,
that induces $\mathcal{N}$ from $\mathcal{M}$. Since $f$ is also
an isomorphism, the argument mutatis mutandis also applies to a dual
transition operator $\tau^{*}$ such that $Z$ is reversible.
\end{proof}
Proposition \ref{prop:7} shows, that the abstract concept of statistical
equivalence has a very precise interpretation given by Statistical
Isomophisms. This motivates to reconsider statistical models as the
objects of a category.
\begin{defn*}[Category of statistical models]
\label{def:category-of-statistical-models}\emph{The category of
statistical models, denoted by }$\mathbf{Stat}$\emph{, consists of:}
\begin{enumerate}
\item[(1)] \emph{a class of objects }$\mathrm{ob}(\mathbf{Stat})$,\emph{ that
comprises all statistical models}
\item[(2)] \emph{a class of morphisms }$\mathrm{hom}(\mathbf{Stat})$,\emph{
that comprises all statistical morphisms}
\end{enumerate}
\end{defn*}

\section{Sufficiency and Completeness in the Category of Statistical Models}

Of course any new mathematical framework at first has to proof it's
usability with respect to established concepts. In the following it is therefore
shown, that the classical concepts of sufficiency and completeness
have a very narrow meaning in the category of statistical models,
respectively given by statistical monomorphisms and statistical epimorphisms.
For the beginning, it is shown, that that statistical morphisms are
closely related to the properties of its underlying measurable function
$T$.
\begin{prop}
\label{prop:8}Let 
\begin{align*}
(X,\,\mathcal{A},\,\mathcal{M}) & \in\mathrm{ob}(\mathbf{Stat})\\
(Y,\,\mathcal{B},\,\mathcal{N}) & \in\mathrm{ob}(\mathbf{Stat})
\end{align*}
with $\mathcal{M}\preccurlyeq\mu$ and $\mathcal{N}\preccurlyeq\nu$
and let
\[
T\colon(X,\,\mathcal{A})\to(Y,\,\mathcal{B})
\]
be a bimeasurable function with:
\[
\mathcal{N}=T_{*}\mathcal{M}
\]
then $(X,\,\mathcal{A},\,\mathcal{M})$ and $(Y,\,\mathcal{B},\,\mathcal{N})$
are statistical equivalent.
\end{prop}

\begin{proof}
Due to lemma \ref{lem:5} $T$ induces a Statistical Morphism $f\colon\mathcal{M}\to\mathcal{N}$
with 
\[
f(P)=T_{*}P,\,\forall P\in\mathcal{M}
\]
which is a.e. unique up to the quotient $\mu/f(\mu)$. Since $T$
is bimeasurable, there also exists an inverse measurable function
$T^{-1}$ with $T^{-1}\circ T=\mathrm{id}_{\mathcal{A}}$. Then also
$T^{-1}$ induces a Statistical Morphism $f^{*}\colon\mathcal{N}\to\mathcal{M}$
with 
\[
f^{*}(Q)=T_{*}^{-1}Q,\,\forall Q\in\mathcal{N}
\]
which is a.e. unique up to the quotient $\nu/f^{*}(\nu)$. Then it
follows, that:
\[
(f^{*}\circ f)(P)=(T^{-1}\circ T)_{*}(P)\stackrel{\mathcal{A}}{=}P,\,\forall P\in\mathcal{M}
\]
and that: 
\[
(f\circ f^{*})(Q)=(T\circ T^{-1})_{*}(Q)\stackrel{\mathcal{B}}{=}Q,\,\forall Q\in\mathcal{N}
\]
By lemma \ref{lem:4.4-2} it follows, that $f$ is a statistical isomorphism
and therefore by proposition \ref{prop:7} that $(X,\,\mathcal{A},\,\mathcal{M})$
and $(Y,\,\mathcal{B},\,\mathcal{N})$ are statistically equivalent.
\end{proof}
The criterion for statistical equivalence, given by proposition \ref{prop:8}
requires the statistic $T$ to be a bimeasurable function, which is
a very strong assumptions. In the purpose to characterise statistical
equivalence by the underlying measurable function, this requirement
therefore has to be weakened in the sense, to only regard events,
that are crucial to the respectively model spaces. For example if
only very few distinct events are necessary for the determination
of the $L^{1}$-identity, then the measurable function $T$ in particular
only has to preserve the distinction of those events and not of all
events, as in the case of a bimeasurable function. In order to formalize
this approach of ``coarse graining'', the injectivity and surjectivity
of a bimeasurable function are respectively weakened by \emph{sufficiency}
and \emph{completeness.}
\begin{defn*}[Sufficiency]
\label{def:Sufficient-statistic}\emph{Let
\[
(X,\,\mathcal{A},\,\mathcal{M})\in\mathrm{ob}(\mathbf{Stat})
\]
 let $(Y,\,\mathcal{B})$ be a measurable space and
\[
T\colon(X,\,\mathcal{A})\to(Y,\,\mathcal{B})
\]
 a measurable function. Let further be $\mathcal{M}_{S}\subseteq\mathcal{M}$
with $\mathcal{M}_{S}\preccurlyeq\mu$, then $T$ is sufficient for
$\mathcal{M}_{S}$, iff: 
\[
\mathrm{p}_{P}(x\mid y)=\mathrm{p}_{\mu}(x\mid y),\,\forall P\in\mathcal{M}_{S}
\]
}
\end{defn*}
Intuitively a measurable function $T\colon(X,\,\mathcal{A})\to(Y,\,\mathcal{B})$,
that is sufficient for a set of probability distributions $\mathcal{M}_{S}\subseteq\mathcal{M}$,
is ``fine enough'', to preserve the $L^{1}$-identity of $\mathcal{M}_{S}$
within its image $T_{*}\mathcal{M}_{S}$, such that for arbitrary
$P,\,Q\in\mathcal{M}_{S}$ it holds that: 
\[
T_{*}P\stackrel{\mathcal{B}}{=}T_{*}Q\Rightarrow P\stackrel{\mathcal{A}}{=}Q
\]
If $T$ is even sufficient for the whole $L^{1}$-space, such that
\[
\mathcal{M}=L^{1}(X,\,\mathcal{A},\,\mu)
\]
with $\mathcal{M}\preccurlyeq\mu$, then $T$ is a bimeasurable function.
This shows, that sufficiency indeed may be regarded as a generalization
of injectivity and furthermore is closely related to statistical monomorphisms.
\begin{lem}
\label{lem:3.6-1}Let 
\begin{align*}
(X,\,\mathcal{A},\,\mathcal{M}) & \in\mathrm{ob}(\mathbf{Stat})\\
(Y,\,\mathcal{B},\,\mathcal{N}) & \in\mathrm{ob}(\mathbf{Stat})
\end{align*}
with $\mathcal{M}\preccurlyeq\mu$ and $\mathcal{N}\preccurlyeq\nu$
and let
\[
T\colon(X,\,\mathcal{A})\to(Y,\,\mathcal{B})
\]
be a measurable function with $T_{*}\mathcal{M}\stackrel{\mathcal{B}}{=}\mathcal{N}$.
Then $T$ is sufficient for $\mathcal{M}$, iff $T$ induces a statistical
monomorphism $f\colon\mathcal{M}\to\mathcal{N}$.
\end{lem}

\begin{proof}
Let $\mathcal{N}_{C}=T_{*}\mathcal{M}\subseteq\mathcal{N}$ then $\mathcal{N}_{C}\preccurlyeq\nu$.
Due to lemma \ref{lem:5} $T$ induces a Statistical Morphism $f\colon\mathcal{M}\to\mathcal{N}_{C}$
and therefore a transition operator 
\[
\tau\colon L^{1}(X,\,\mathcal{A},\,\mathcal{M})\to L^{1}(Y,\,\mathcal{B},\,\mathcal{N}_{C})
\]
with $\tau(\mathcal{M})=\mathcal{N}_{C}$, which is a.e. unique up
to $\mu/\tau(\mu)$. In particular the $\mathrm{p}_{P}(y\mid x)$
do not depend on the choice of $P\in\mathcal{M}$. ``$\Longrightarrow$''
Let now be $T$ sufficient for $\mathcal{M}$, then by definition
also the induced dual conditional probabilities $\mathrm{p}_{P}(x\mid y)$
do not depend on the choice of $P\in\mathcal{M}$. Due to the generalized
principle of detailed balance there exists a reversible Markov process
from $X$ to $Y$, that pushes $\mathcal{M}$ to $\mathcal{N}_{C}$,
such that $(X,\,\mathcal{A},\,\mathcal{M})$ and $(Y,\,\mathcal{B},\,\mathcal{N}_{C})$
are statistical equivalent. Then by proposition \ref{prop:7} it follows,
that $f\mid_{\mathcal{N}_{C}}$ is a statistical isomorphism and $f$
is a statistical monomorphism. ``$\Longleftarrow$'' Let $f$ be
a statistical monomorphism, then $f\mid_{\mathcal{N}_{C}}$ is a statistical
isomorphism and proposition \ref{prop:7} postulates that $(X,\,\mathcal{A},\,\mathcal{M})$
and $(Y,\,\mathcal{B},\,\mathcal{N}_{C})$ are statistical equivalent
such that there exists a reversible Markov process from $X$ to $Y$,
that pushes $\mathcal{M}$ to $\mathcal{N}_{C}$. Then the generalized
principle of detailed balance postulates, that the $\mathrm{p}_{P}(x\mid y)$
do not depend on the choice of $P\in\mathcal{M}$, such that $T$
is sufficient for $\mathcal{M}$.
\end{proof}
Conversely also the question arises, if a measurable function $T$
is ``fine enough'', to generate the $L^{1}$-identity over its codomain.
Let therefore $(X,\,\mathcal{A})$ be a measurable space, $(Y,\,\mathcal{B},\,\mathcal{N})$
a statistical model with $\mathcal{N}\preccurlyeq\nu$ and 
\[
T\colon(X,\,\mathcal{A})\to(Y,\,\mathcal{B})
\]
a measurable function. Then for any $\rho\in L^{1}(Y,\,\mathcal{B},\,\nu)$
an approximation of $\rho$ is given by the conditional expectation
$\mathrm{E}_{\nu}(\rho\mid T(\mathcal{A}))$ and $T$ generates the
$L^{1}$-identity over $(Y,\,\mathcal{B},\,\mathcal{N})$, if for
any 
\[
\rho\in L^{1}(Y,\,\mathcal{B},\,\nu)
\]
 with 
\[
\mathrm{E}_{\nu}(\rho\mid T(\mathcal{A}))\stackrel{\mathcal{B}}{=}0_{\mathcal{B}}
\]
 it follows, that $\rho\stackrel{\mathcal{B}}{=}0_{\mathcal{B}}$.
This provides the notation of completeness.
\begin{defn*}[Completeness]
\label{def:Complete-statistic}\emph{Let $(X,\,\mathcal{A})$ be
a measurable space, $(Y,\,\mathcal{B},\,\mathcal{N})\in\mathrm{ob}(\mathbf{Stat})$
and 
\[
T\colon(X,\,\mathcal{A})\to(Y,\,\mathcal{B})
\]
 a measurable function. Let further be $\mathcal{N}_{C}\subseteq\mathcal{N}$
with $\mathcal{N}_{C}\preccurlyeq\nu$, then $T$ is termed complete
for $\mathcal{N}_{C}$, iff for all 
\[
\rho\in L^{1}(Y,\,\mathcal{B},\,\nu)
\]
 with 
\[
\mathrm{E}_{\nu}(\rho\mid T(\mathcal{A}))\stackrel{\mathcal{B}}{=}0_{\mathcal{B}}
\]
 it follows, that $\rho\stackrel{\mathcal{B}}{=}0_{\mathcal{B}}$.}
\end{defn*}
In particular if a measurable function $T$ is complete for $\mathcal{N}_{C}$,
then for arbitrary $P,\,Q\in\mathcal{N}_{C}$ with $P\notin\mathrm{id}_{\mathcal{B}}(Q)$
there exists an $A\in\mathcal{A}$, such that: 
\[
P(T(A))\ne Q(T(A))
\]
By assuming an underlying statistical model $(X,\,\mathcal{A},\,\mathcal{M})$
and $\mathcal{N}$ to be induced by $T$, such that $\mathcal{N}=T_{*}\mathcal{M}$,
then the condition may also be pulled back to $\mathcal{M}$. Then
$T$ is complete for $\mathcal{N}_{C}$, if for arbitrary for $P,\,Q\in\mathcal{N}_{C}$
with $P\notin\mathrm{id}_{\mathcal{B}}(Q)$ there exists an $A\in\mathcal{A}$,
such that 
\[
(T^{*}P)(A)\ne(T^{*}Q)(A)
\]
This however is equivalent to the claim, that $T$ generates the $L^{1}$-identity
over $\mathcal{N}_{C}$, such that for all $Q\in\mathcal{N}_{C}$,
there exists an $P\in\mathcal{M}$, such that $T_{*}P\stackrel{\mathcal{B}}{=}Q$.
This shows, that completeness generalizes surjectivity and furthermore
is closely related to Statistical Epimorphisms.
\begin{lem}
\label{lem:3.6-1-1}Let 
\begin{align*}
(X,\,\mathcal{A},\,\mathcal{M}) & \in\mathrm{ob}(\mathbf{Stat})\\
(Y,\,\mathcal{B},\,\mathcal{N}) & \in\mathrm{ob}(\mathbf{Stat})
\end{align*}
with $\mathcal{M}\preccurlyeq\mu$ and $\mathcal{N}\preccurlyeq\nu$
and 
\[
T\colon(X,\,\mathcal{A})\to(Y,\,\mathcal{B})
\]
 a measurable function with $T_{*}\mathcal{M}\stackrel{\mathcal{B}}{=}\mathcal{N}$.
Then $T$ is complete for $\mathcal{N}$, iff $T$ induces a Statistical
Epimorphism $f\colon\mathcal{M}\to\mathcal{N}$.
\end{lem}

\begin{proof}
``$\Longrightarrow$'' Since $T_{*}\mathcal{M}\stackrel{\mathcal{B}}{=}\mathcal{N}$
it follows, that $T_{*}\mathcal{M}\subseteq\mathcal{N}$. Since $T$
is complete for $\mathcal{N}$, for all $Q\in\mathcal{N}$ there exists
an $P\in\mathcal{M}$, such that $T_{*}P\stackrel{\mathcal{B}}{=}Q$
and therefore $\mathcal{N}\subseteq T_{*}\mathcal{M}$, such that
$T_{*}\mathcal{M}=\mathcal{N}$. Then due to lemma \ref{lem:5} $T$
induces a Statistical Morphism $f\colon\mathcal{M}\to\mathcal{N}$
with $f(\mathcal{M})=T_{*}\mathcal{M}$ and since 
\[
\mathrm{img}(f)=T_{*}\mathcal{M}\stackrel{\mathcal{B}}{=}\mathcal{N}
\]
 it follows, that $f$ is a Statistical Epimorphism. ``$\Longleftarrow$''
Let $f$ be a Statistical Epimorphism, then $\mathrm{img}(f)\stackrel{\mathcal{B}}{=}\mathcal{N}$
where 
\[
f(\mathcal{M})=T_{*}\mathcal{M}\stackrel{\mathcal{B}}{=}\mathcal{N}
\]
Then for any $Q\in\mathcal{N}$ there exists an $P\in\mathcal{M}$,
such that $T_{*}P\stackrel{\mathcal{B}}{=}Q$ and $T$ is complete
for $\mathcal{N}$.
\end{proof}
The properties of sufficiency for $\mathcal{M}$ and completeness
for $\mathcal{N}$ are still satisfied if they are only required for
the subsets $\bigslant{\mathcal{M}}{\mathrm{id}_{\mathcal{A}}}\subseteq\mathcal{M}$
and $\bigslant{\mathcal{N}}{\mathrm{id}_{\mathcal{B}}}\subseteq\mathcal{N}$,
that generate the respective $L^{1}$-identity. Intuitively this represents
the property that for the statistical equivalence of statistical models
it simply doesn't matter if the model spaces are generated by distribution
assumptions, that are observational indistinguishable, as long as
the distinguishable probability distributions are distinguishable
in both statistical models. This allows to characterise statistical
equivalence by the underlying measurable function.
\begin{thm}
\label{thm:3.4}Let 
\begin{align*}
(X,\,\mathcal{A},\,\mathcal{M}) & \in\mathrm{ob}(\mathbf{Stat})\\
(Y,\,\mathcal{B},\,\mathcal{N}) & \in\mathrm{ob}(\mathbf{Stat})
\end{align*}
with $\mathcal{M}\preccurlyeq\mu$ and $\mathcal{N}\preccurlyeq\nu$.
Then the following statements are equivalent:
\begin{enumerate}
\item[\emph{(i)}] There exists a measurable function 
\[
T\colon(X,\,\mathcal{A})\to(Y,\,\mathcal{B})
\]
with $T_{*}\mathcal{M}\stackrel{\mathcal{B}}{=}\mathcal{N}$, such
that $T$ is sufficient for $\bigslant{\mathcal{M}}{\mathrm{id}_{\mathcal{A}}}$
and $T$ is complete for $\bigslant{\mathcal{N}}{\mathrm{id}_{\mathcal{B}}}$
\item[\emph{(ii)}] $(X,\,\mathcal{A},\,\mathcal{M})$ and $(Y,\,\mathcal{B},\,\mathcal{N})$
are statistical equivalent
\item[\emph{(iii)}] There exists a statistical isomorphism $f\colon\mathcal{M}\to\mathcal{N}$
\end{enumerate}
\end{thm}

\begin{proof}
``$(\text{i})\Rightarrow(\text{ii})$'' Since $T$ is sufficient
for $\bigslant{\mathcal{M}}{\mathrm{id}_{\mathcal{A}}}$, it is also
sufficient for $\mathcal{M}$ and therefore by lemma \ref{lem:3.6-1}
it follows, that $T$ induces a statistical monomorphism $f\colon\mathcal{M}\to\mathcal{N}$.
Furthermore since $T$ is complete for $\bigslant{\mathcal{N}}{\mathrm{id}_{\mathcal{B}}}$,
it is also complete for $\mathcal{N}$ and therefore by lemma \ref{lem:3.6-1-1}
and the a.e. uniqueness of $f$ it follows, that $f$ is also a Statistical
Epimorphism and therefore a statistical isomorphism. Then due to proposition \ref{prop:7}
$(X,\,\mathcal{A},\,\mathcal{M})$ and $(Y,\,\mathcal{B},\,\mathcal{N})$
are statistical equivalent. ``$(\text{ii})\Rightarrow(\text{i})$''
Let $(X,\,\mathcal{A},\,\mathcal{M})$ and $(Y,\,\mathcal{B},\,\mathcal{N})$
be statistical equivalent, then proposition \ref{prop:7} postulates the
existence of a statistical isomorphism $f\colon\mathcal{M}\to\mathcal{N}$.
``$(\text{ii})\Leftrightarrow(\text{iii})$'' Has already been proven
by proposition \ref{prop:7}.
\end{proof}

\section{\label{subsec:Topological-structure}Topological Statistical Models}

Apart of an underlying reversible Markov process, it would be pleasant
to characterise statistical equivalence directly by a proximity relationship
of the model spaces. In section \ref{subsec:Statistical-Equivalence}
this observable structure has intuitively been introduced by a pairwise
comparison of probability distributions, with respect to the collection
of events in the sample space. What is common to them, is the capability
to induce a unique topology to the model space, that characterises
the statistical inference with respect to the evaluations. The intuition behind
the concept of statistical models is to provide a framework for \emph{statistical
inference}, which eventually allows to derive conclusions about the
``true model'' by coincidences between observable and structural
beliefs. Consequentially the fundamental ability to derive any conclusions
depends on the distinguishability of distribution assumptions with respect to
realizations of finite random samples.

More precisely for a hypothesis Space $\mathcal{H}$ any distribution assumptions $H_{P},\,H_{Q}\in\mathcal{H}$ are \emph{observable distinguishable}, iff their induced probability distributions $P,\,Q\in\mathcal{M}$ are distinguishable by an event $A\in\Sigma$, such that $P(A)\ne Q(A)$. Beyond the observable distinguishability
however, the capability to decide which distribution assumption is
``closer'' to an observation, requires the existence of a proximity
structure within the model space. The \emph{observable structure}
of a statistical model therefore provides the foundation of observation
based statistical inference. Although for arbitrary statistical models,
there is no natural selection of an observable structure, this deficiency
may be resolved, by requiring the induced probability distributions
to admit a density function. This restriction yields the notation
of \emph{continuous statistical models}.
\begin{defn*}[Continuous statistical model]
\label{def:Continous-statistical-model}\emph{Let $(S,\,\Sigma,\,\mathcal{M})$
be a statistical model. Then $(S,\,\Sigma,\,\mathcal{M})$ is termed
continuous iff (i) the sample space $(S,\,\Sigma)$ is a Borel-space
and (ii) all $P\in\mathcal{M}$ are absolutely continuous over $(S,\,\Sigma)$.
}\textbf{Remark}:\emph{ In the following all statistical models are
assumed to be continuous unless stated otherwise.}
\end{defn*}
Due to its definition, the model space of a continuous statistical
model $(S,\,\Sigma,\,\mathcal{M})$ is embedded within the space of
Lebesgue measurable functions over $(S,\,\Sigma)$, which naturally
induces the $L_{1}$-topology to $\mathcal{M}$. The underlying statistic
$T\colon(\Omega,\,\mathcal{F},\,\mathrm{P})\to(S,\,\Sigma)$ then
allows to pull back the topology to $\mathcal{H}$ and therefore endows
the distribution assumptions with a proximity structure. In particular
distribution assumptions are observable distinguishable iff they are
distinguishable within the induced $L_{1}$-topology. However in order
to quantify the proximity between probability assumptions and observations
also the respective ``$L_{1}$-neighbourhoods'' of empirical distributions
have to be pulled back. A technical difficulty arises from
the fact that empirical distributions are not Lebesgue measurable.
Nevertheless since the empirical distributions are at least Bochner
measurable, it suffices to (i) extend $L^{1}(S,\,\Sigma)$ by the
empirical distributions of a finite random sample $X$ and (ii) derive
the observable structure from the extended topology. In this purpose
let $X$ be an $(S,\,\Sigma)$-valued finite random sample of length
$N$ and $\mathcal{E}(X)$ the set of all empirical distributions
$\mathrm{P}_{n}(X)$ of $X$ with $n\leq N$. Let further be $d_{1}$
the $L_{1}$-distance over $L^{1}(S,\,\Sigma)$, such that $d_{1}(\mu,\,\nu)\coloneqq\|\mu-\nu\|_{1}$
for all $\mu,\,\nu\in L^{1}(S,\,\Sigma)$ and $L_{X}^{1}(S,\,\Sigma)$
the smallest Bochner space over $(S,\,\Sigma)$, that covers $L^{1}(S,\,\Sigma)$
as well as $\mathcal{E}(X)$. Then due to the finite number of jump
discontinuities within the empirical distributions, the $L_{1}$-distance
$d_{1}$ may naturally be extended to $L_{X}^{1}(S,\,\Sigma)$ by
a formal integration by parts. Consequentially the extended $L_{1}$-topology,
which is generated by $d_{1}$ in particular makes $L_{X}^{1}(S,\,\Sigma)$
a topological vector space, which covers $\mathcal{M}$ and $\mathcal{E}(X)$
and therefore provides a natural choice for an observable structure
of continuous statistical models. The next step towards an effective
approximation of the sample distribution concerns the convergence
rate of the empirical distributions. Although the $L_{1}$-convergence
of empirical distributions is assured in their asymptotic limit, the
$L_{1}$-topology may be too strong, to capture the remaining uncertainty
of finite random samples. This is due to the fact that the $L_{1}$-topology
considers every single aspect of the sample distribution by the complete
evaluation of individual events. Therefore if some aspects are considered
to be ``more important'' than others, then a more commensurate topology
may be obtained by a restriction of the evaluation to these aspects.
This restriction is performed by\emph{ estimands}.
\begin{defn*}[Continues estimand]
\label{def:Continous-estimand}\emph{Let $(S,\,\Sigma,\,\mathcal{M})$
be a continuous statistical model and $X$ a finite random sample
in $S$ of length $n$, which generates the $\sigma$-Algebra $\mathcal{A}^{n}$.
Then a mapping $\epsilon\colon\mathcal{M}\times\mathcal{A}^{n}\rightarrow\mathbb{R}$
is a continuous estimand, iff (i) $\epsilon$ is $L^{1}$-continuous
in its first argument and (ii) $\epsilon$ is $\mathcal{A}^{n}$-symmetric
in its second argument.}
\end{defn*}
Estimands provide the opportunity to restrict an evaluation of probability
distributions to assorted aspects. As this evaluation is performed
with respect to given realizations $A\in\mathcal{A}^{n}$ they conditionally
dependent on the respective realizations and therefore are given by
the the notation $\epsilon(P\mid X\in A)\coloneqq\epsilon(P,\,A)$.
This notation implicates the abbreviations ``$\epsilon(P\mid A)$'',
to emphasize a value in $\mathbb{R}$ and ``$\epsilon(P\mid X)$'',
to emphasize a function in $L_{X}^{1}(S,\,\Sigma)$. Thereby the requirement
of the function $\epsilon(P\mid X)\colon A\mapsto\epsilon(P\mid A)$
to be symmetric with respect to $A=\{A_{i}\}\in\mathcal{A}^{n}$ assures permutation
invariance and therefore preserves the independence of the individual
observations $A_{i}$. Furthermore the requirement of the operator
$\epsilon\colon P\mapsto\epsilon(P\mid X)$ to be $L_{1}$-continuous
induces a topology within $L_{X}^{1}(S,\,\Sigma)$, by the quotient
topology with respect to its kernel. In particular this induced \emph{$\epsilon$-topology}
is identical to the $L_{1}$-topology iff $\epsilon$ is an $L_{1}$-homeomorphism.
Consequently the question for the existence of continous estimands
arises.
\begin{lem}
\label{lem:3.7-1}Let\emph{
\[
(S,\,\Sigma,\,\mathcal{M})\in\mathrm{ob}(\mathbf{Stat})
\]
}$X$ a finite sample over $(S,\,\Sigma)$, $\epsilon$ an estimand
of $\mathcal{M}$ over $X$. Then for any $P\in\mathcal{M}$ there
exists a unique coarsest topology $\mathcal{T}_{P}$ over $\mathcal{M}$,
that makes $\epsilon$ continuous w.r.t. $P$. 
\end{lem}

\begin{proof}
Let $\mathcal{A}$ be the induced $\sigma$-algebra of the sample
$X$. Then the definition of $\epsilon$ assures the existence of
a $\sigma$-finite measure $\mu$ over $(X,\,\mathcal{A})$ with $\mathcal{M}\subseteq L^{1}(S,\,\Sigma,\,\mu)$,
such that for any $Q\in\mathcal{M}$ the function $\epsilon_{Q}\colon A\mapsto\epsilon(Q\mid A)$
is absolutely continuous w.r.t. $\mu$ and therefore $\epsilon_{Q}\in L^{1}(S,\,\Sigma,\,\mu)$.
This allows the definition of a distance by: 
\[
d(P,\,Q)\coloneqq\|\epsilon_{P}-\epsilon_{Q}\|_{1}
\]
Let $S_{P}=\{d(P,\,Q)\mid Q\in\mathcal{M}\}$, then 
\[
\epsilon_{P},\,\epsilon_{Q}\in L^{1}(S,\,\Sigma,\,\mu)
\]
assures the existence of $a,\,b\in\mathbb{R}$ with $S_{P}\subseteq[a,\,b]$.
Let $\mathcal{T}_{S}$ be the subspace topology of $S_{P}$ in $\mathbb{R}$,
then for any $V\in\mathcal{T}_{S}$ let 
\[
U_{V}=\{Q\in\mathcal{P}\mid d(P,\,Q)\in V\}
\]
Then $\mathcal{T}_{P}=\{U_{V}\mid V\in\tau_{S}\}$ defines a topology
over $\mathcal{M}$, which is second countable, since $\mathcal{T}_{S}$
has a countable base. Furthermore any $Q\in\mathcal{P}$ is topologically
distinguishable from $P$ iff there exists an $A\in\Sigma$ with 
\[
\epsilon(P\mid A)\ne\epsilon(Q\mid A)
\]
\end{proof}
\begin{thm}[\emph{Initial Theorem}]
\emph{\label{thm:Initial-theorem}} Let $(S,\,\Sigma,\,\mathcal{M})$
be a statistical model, $X$ a finite sample over $(S,\,\Sigma)$
and $\epsilon$ an estimand of $\mathcal{M}$ over $X$. Then there
exists a unique coarsest topology $\mathcal{T}$ over $\mathcal{M}$,
that makes $\epsilon$ continuous.
\end{thm}

\begin{proof}
Due to lemma \ref{lem:3.7-1} for any $P\in\mathcal{M}$ the evaluation
$\epsilon$ induces a topology $\mathcal{T}_{P}$ in $\mathcal{M}$,
which inherits the property of a countable base. Let $\mathcal{U}$
be the unification of all topologies $\mathcal{T}_{P}$ and $\mathcal{T}$
the coarsest topology over $\mathcal{M}$, that covers $\mathcal{U}$.
Then $\mathcal{T}$ has a countable base and is the coarsest topology
over $\mathcal{M}$, that preserves the continuity of $\epsilon$
within $P$. In particular $\mathcal{T}$ preserves the distinction
of any $P,\,Q\in\mathcal{M}$ w.r.t. $\epsilon$, since $P$ and $Q$
are distinguishable in $\mathcal{T}$ iff there exists an $A\in\Sigma$
with $\epsilon(P\mid A)\ne\epsilon(Q\mid A)$.
\end{proof}
The Initial Theorem of Topological Statistical Theory for any given
estimand assures the existence of a unique coarsest topology, that
makes it continuous. Without a given estimand, however, one might
of course also like to be able to derive a proximity structure. A
canonical choice for an $L_{1}$-homeomorphic estimand can be obtained
by the ``evaluation of all individual events''. Therefore let $X$
be a finite random sample of length $n$, that generates the sigma
algebra $\mathcal{A}^{n}$. Then any $P\in\mathcal{M}$ induces a
canonical probability distribution over $\mathcal{A}^{n}$ by its
product measure, which is known as the \emph{likelihood} \emph{function}.

\begin{defn*}[Likelihood]
\label{def:Likelihood}\emph{Let $(S,\,\Sigma,\,\mathcal{M})$ be
a continuous statistical model and $X$ a finite random sample in
$S$ of length $n$, which generates the $\sigma$-Algebra $\mathcal{A}^{n}$.
Then the likelihood of $P\in\mathcal{M}$ w.r.t. $A\in\mathcal{A}^{n}$
is given by: 
\[
\mathrm{L}(P\mid X\in A)\coloneqq\prod_{i=1}^{n}P(X_{i}\in A_{i})
\]
}
\end{defn*}
Due to its definition the likelihood function is symmetric with regard
to realizations $A\in\mathcal{A}^{n}$. Furthermore the operator $\mathrm{L}\colon P\mapsto\mathrm{L}(P\mid X\in A)$
is $L_{1}$-continuous and bijective with respect to its image within $L_{1}(S^{n},\,\mathcal{A}^{n})$
and since conversely also the projection $\pi_{1}\colon\mathrm{L}(P\mid X\in A)\mapsto P(X_{1}\in A_{1})$
is $L_{1}$-continuous and bijective it follows that $\mathrm{L}$
is an $L_{1}$-homeomorphism. This shows, that the likelihood function
is a continuous estimand which induces the $L_{1}$-topology. In the
following it is shown, that the
\begin{cor}
\label{cor:3.2}Let $(S,\,\Sigma,\,\mathcal{M})$ be a statistical
model. Then the likelihood function induces the unique coarsest topology
$\mathcal{T}$ over $\mathcal{M}$, that makes $\mathcal{M}$ continuous
w.r.t. any estimands over $(S,\,\Sigma)$. 
\end{cor}

\begin{proof}
By choosing the likelihood function $\mathrm{L}$ as an estimand theorem \ref{thm:Initial-theorem}
postulates the existence of a unique coarsest topology $\mathcal{T}$,
that makes $\mathrm{L}$ continuous. Then any $P,\,Q\in\mathcal{M}$
are topological distinguishable in $\mathcal{T}$, iff there exists
an $A\in\Sigma$, such that 
\[
\mathrm{L}(P\mid A)\ne\mathrm{L}(Q\mid A)
\]
By applying the definition of the likelihood function this is equivalent
to the condition, that $P(A)\ne Q(A)$, which proves that $\mathcal{T}$
preserves the observational distinguishability of probability distributions
in $\mathcal{M}$ over $\Sigma$.
\end{proof}
\begin{defn*}[Canonical topology]
\label{def:Observable-topology}\emph{Let $(S,\,\Sigma,\,\mathcal{M})$
be a statistical model and $\mathcal{T}$ the unique coarsest topology,
that makes $\mathcal{M}$ continuous w.r.t. $(S,\,\Sigma)$. Then
$\mathcal{T}$ is termed the canonical topology of $\mathcal{M}$
over $(S,\,\Sigma)$.}
\end{defn*}
The canonical topology of a statistical model, describes a topology
of the model space, that preserves the ``proximity'' of probability
distributions with respect to their evaluation over the sample space.
Corollary \ref{cor:3.2} thereby postulates, that the canonical topology
is unique and therefore an intrinsic structure of a statistical model.
This property gives rise to study the statistical equivalence of statistical
models in terms of topological spaces. This description provides the
notation of \emph{topological statistical models}.
\begin{defn*}[Topological statistical model]
\label{def:Topological-statistical-model}\emph{Let 
\[
(S,\,\Sigma,\,\mathcal{M})\in\mathrm{ob}(\mathbf{Stat})
\]
and let $\mathcal{T}$ be a topology of $\mathcal{M}$ over $S$.
Then the 4-tuple 
\[
(S,\,\Sigma,\,\mathcal{M},\,\mathcal{T})
\]
is termed a topological statistical model and canonical, iff $\mathcal{T}$
is the canonical topology of $\mathcal{M}$ over $(S,\,\Sigma)$.}
\end{defn*}
In the purpose to describe statistical equivalence in the context
of topological statistical models it comes naturally to mind to utilize
homeomorphisms. Thereby these homeomorphism are only required to imply
observational distinguishable probability distributions. In the context
of Statistical Morphisms, in theorem \ref{thm:3.4} this circumstance
has mutatis mutandis been satisfied by a formulation, that uses the
quotient spaces $\bigslant{\mathcal{M}}{\mathrm{id}_{\mathcal{A}}}$
and $\bigslant{\mathcal{N}}{\mathrm{id}_{\mathcal{B}}}$. These quotients
naturally extend to quotient topologies of their respective observable
topologies by $\bigslant{\mathcal{T}_{\mathcal{A}}}{\mathrm{id}_{\mathcal{A}}}$
and $\bigslant{\mathcal{T}_{\mathcal{B}}}{\mathrm{id}_{\mathcal{B}}}$.
Since the probability distributions within a model space however are
observational identical iff they are topological indistinguishable,
the corresponding quotient spaces of $(\mathcal{M},\,\mathcal{T}_{\mathcal{A}})$
and $(\mathcal{N},\,\mathcal{T}_{\mathcal{B}})$ may also directly
be defined by their topological indistinguishability. This provides
the notation of \emph{Kolmogorov quotients.}
\begin{defn*}[Kolmogorov quotient]
\label{def:Kolmogorov-quotient}\emph{Let $(\mathcal{M},\,\mathcal{T})$
be a topological space. Then the Kolmogorov quotient $\mathrm{KQ}(\mathcal{M},\,\mathcal{T})$
denotes the quotient space of $(\mathcal{M},\,\mathcal{T})$ w.r.t.
the equivalence of topological indistinguishabiltiy.}
\end{defn*}
\begin{lem}
\label{lem:3.8-1}Let $(X,\,\mathcal{A},\,\mathcal{P},\,\mathcal{T}_{\mathcal{A}})$
and $(Y,\,\mathcal{B},\,\mathcal{Q},\,\mathcal{T}_{\mathcal{B}})$
be topological statistical models and $T\colon(X,\,\mathcal{A})\to(Y,\,\mathcal{B})$
a measurable function. Then the following statements are equivalent: 
\begin{enumerate}
\item[\emph{(i)}] There exists a Statistical Morphism 
\[
f\colon\mathcal{P}\to\mathcal{Q}
\]
\item[\emph{(ii)}] There exists a linear operator 
\[
f\colon L^{1}(X,\,\mathcal{A},\,\mathcal{P})\to L^{1}(Y,\,\mathcal{B},\,\mathcal{Q})
\]
such that 
\[
f\colon\mathrm{KQ}(\mathcal{P},\,\mathcal{T}_{\mathcal{A}})\to\mathrm{KQ}(\mathcal{Q},\,\mathcal{T}_{\mathcal{B}})
\]
is continuous
\end{enumerate}
\end{lem}

\begin{proof}
``$\mathrm{(i)}\,\Longrightarrow\mathrm{(ii)}$'': Let $f$ be a
Statistical Morphism. Then the definition of $f$ implicates the existence
of a transition operator 
\[
\tau\colon L^{1}(X,\,\mathcal{A},\,\mathcal{P})\to L^{1}(Y,\,\mathcal{B},\,\mathcal{Q})
\]
such that $f=\tau|_{\mathcal{P}}$. Since $\tau$ is a continuous
linear operator it follows, that for $g\coloneqq\tau$ also 
\[
g\colon\mathrm{KQ}(\mathcal{P},\,\mathcal{T}_{\mathcal{A}})\to\mathrm{KQ}(\mathcal{Q},\,\mathcal{T}_{\mathcal{B}})
\]
 is continuous. ``$\mathrm{(ii)}\,\Longrightarrow\mathrm{(i)}$'':
Let 
\[
f\colon\mathrm{KQ}(\mathcal{P},\,\mathcal{T}_{\mathcal{A}})\to\mathrm{KQ}(\mathcal{Q},\,\mathcal{T}_{\mathcal{B}})
\]
 be continuous. By assuming, that $f$ would be not be continuously
linearised with respect to $T$, then $f$ would not be continuous,
which contradicts to the preliminary condition. Let therefore 
\[
g\colon\mathrm{KQ}(\mathcal{Q},\,\mathcal{T}_{\mathcal{B}})\to\mathrm{KQ}(\mathcal{Q},\,\mathcal{T}_{\mathcal{B}})
\]
 be a homeomorphism, such that $g\circ f$ commutes with $T$, $\pi_{\mathcal{A}}\colon\mathcal{P}\to\bigslant{\mathcal{P}}{\mathrm{id}_{\mathcal{A}}}$
a natural projection and $h\coloneqq g\circ f\circ\pi_{\mathcal{A}}$.
Then $h$ commutes with $T$, $\mathrm{dom}(h)=\mathcal{P}$ and $\mathrm{img}(h)=\mathcal{Q}$,
such that $h\colon\mathcal{P}\to\mathcal{Q}$ is a Statistical Morphism.
\end{proof}
\begin{defn*}[Kolmogorov equivalence]
\label{def:Kolmogorov-equivalence}\emph{Let $(\mathcal{P},\,\mathcal{T}_{\mathcal{A}})$
and $(\mathcal{Q},\,\mathcal{T}_{\mathcal{B}})$ be topological spaces,
then $(\mathcal{P},\,\tau_{\mathcal{A}})$ and $(\mathcal{Q},\,\tau_{\mathcal{B}})$
are termed Kolmogorov equivalent, iff there exists a homeomorphism
\[
f\colon\mathrm{KQ}(\mathcal{P},\,\mathcal{T}_{\mathcal{A}})\to\mathrm{KQ}(\mathcal{Q},\,\mathcal{T}_{\mathcal{B}})
\]
}
\end{defn*}
\begin{thm}
\label{thm:3.6}For the statistical population $(\Omega,\,\mathcal{F},\,\mathrm{P})$
let $(X,\,\mathcal{A},\,\mathcal{P})$ and $(Y,\,\mathcal{B},\,\mathcal{Q})$
be statistical models and $\mathcal{T}_{\mathcal{A}}$ and $\mathcal{T}_{\mathcal{B}}$
their corresponding canonical topologies. Then the following statements
are equal: 
\begin{enumerate}
\item[\emph{(i)}] $(X,\,\mathcal{A},\,\mathcal{P})$ and $(Y,\,\mathcal{B},\,\mathcal{Q})$
are statistical equivalent
\item[\emph{(ii)}] $\mathrm{KQ}(\mathcal{P},\,\mathcal{T}_{\mathcal{A}})$ and $\mathrm{KQ}(\mathcal{Q},\,\mathcal{T}_{\mathcal{B}})$
are linear isomorphic and homeomorphic
\end{enumerate}
\end{thm}

\begin{proof}
``$\mathrm{(i)}\,\Longrightarrow\mathrm{(ii)}$'': Let $(X,\,\mathcal{A},\,\mathcal{P})$
and $(Y,\,\mathcal{B},\,\mathcal{Q})$ be statistical equivalent.
Then there exists a Statistical Morphism $\mathfrak{T}\colon\mathcal{P}\to\mathcal{Q}$,
as well as a dual Statistical Morphism $\mathfrak{T}^{*}\colon\mathcal{P}\to\mathcal{Q}$,
such that: 
\[
(\mathfrak{T}^{*}\circ\mathfrak{T})(P)\in\mathrm{id}_{\mathcal{A}}(P),\,\forall P\in\mathcal{P}
\]
and 
\[
(\mathfrak{T}^{*}\circ\mathfrak{T})(Q)\in\mathrm{id}_{\mathcal{B}}(Q),\,\forall Q\in\mathcal{Q}
\]
For 
\[
f\colon\mathrm{KQ}(\mathcal{P},\,\mathcal{T}_{\mathcal{A}})\to\mathrm{KQ}(\mathcal{Q},\,\mathcal{T}_{\mathcal{B}})
\]
\[
f(P)\coloneqq\mathfrak{T}(P),\,\forall P\in\bigslant{\mathcal{P}}{\mathrm{id}_{\mathcal{A}}}
\]
and 
\[
f^{*}\colon\mathrm{KQ}(\mathcal{Q},\,\mathcal{T}_{\mathcal{B}})\to\mathrm{KQ}(\mathcal{P},\,\mathcal{\tau}_{\mathcal{A}})
\]
\[
f^{*}(Q)\coloneqq\mathfrak{T}^{*}(Q),\,\forall Q\in\bigslant{\mathcal{Q}}{\mathrm{id}_{\mathcal{B}}}
\]
it then follows, that: 
\[
(f^{*}\circ f)(P)=P,\,\forall P\in\bigslant{\mathcal{P}}{\mathrm{id}_{\mathcal{A}}}
\]
\[
(f\circ f^{*})(Q)=Q,\,\forall Q\in\bigslant{\mathcal{Q}}{\mathrm{id}_{\mathcal{B}}}
\]
such that $f$ is invertible. Since due to lemma \ref{lem:3.8-1}
$f$, as well as its inverse $f^{*}$ are continuous, $f$ is a homeomorphism.
This proves, that there exists a homeomorphism 
\[
f\colon\mathrm{KQ}(\mathcal{P},\,\mathcal{T}_{\mathcal{A}})\to\mathrm{KQ}(\mathcal{Q},\,\mathcal{T}_{\mathcal{B}})
\]
and therefore that $(\mathcal{P},\,\mathcal{T}_{\mathcal{A}})$ and
$(\mathcal{Q},\,\mathcal{T}_{\mathcal{B}})$ are Kolmogorov equivalent
.

``$\mathrm{(ii)}\,\Longrightarrow\mathrm{(i)}$'': Let now$(\mathcal{P},\,\mathcal{T}_{\mathcal{A}})$
and $(\mathcal{Q},\,\mathcal{T}_{\mathcal{B}})$ be Kolmogorov equivalent,
then there exists a homeomorphism 
\[
f\colon\mathrm{KQ}(\mathcal{P},\,\mathcal{T}_{\mathcal{A}})\to\mathrm{KQ}(\mathcal{Q},\,\mathcal{T}_{\mathcal{B}})
\]
Let 
\[
T_{X}\colon(\Omega,\,\mathcal{F},\,\mathrm{P})\to(X,\,\mathcal{A})
\]
\[
T_{Y}\colon(\Omega,\,\mathcal{F},\,\mathrm{P})\to(Y,\,\mathcal{B})
\]
be the statistics, that generate the sample spaces. Then $T\coloneqq T_{Y}\circ T_{X}^{-1}$
is a measurable function with: 
\[
T\colon(X,\,\mathcal{A})\to(Y,\,\mathcal{B})
\]
Then due to the proof of lemma \ref{lem:3.8-1} there exists a homeomorphism
\[
g\colon\mathrm{KQ}(\mathcal{Q},\,\mathcal{T}_{\mathcal{B}})\to\mathrm{KQ}(\mathcal{Q},\,\mathcal{T}_{\mathcal{B}})
\]
such that $g\circ f$ defines a Statistical Morphism 
\[
g\circ f\colon\bigslant{\mathcal{P}}{\mathrm{id}_{\mathcal{A}}}\to\bigslant{\mathcal{Q}}{\mathrm{id}_{\mathcal{B}}}
\]
which commutes with $T$. Since $f$, as well as $g$ are homeomorphisms
also $g\circ f$ is a homeomorphism and its inverse $(g\circ f)^{-1}$
is continuous. Let now
\[
\pi_{\mathcal{A}}\colon\mathcal{P}\to\bigslant{\mathcal{P}}{\mathrm{id}_{\mathcal{A}}}
\]
\[
\pi_{\mathcal{B}}\colon\mathcal{Q}\to\bigslant{\mathcal{Q}}{\mathrm{id}_{\mathcal{B}}}
\]
be natural projections that satisfy: 
\[
\mathfrak{T}\coloneqq g\circ f\circ\pi_{\mathcal{A}}
\]
\[
\mathfrak{T}^{*}\coloneqq(g\circ f)^{-1}\circ\pi_{\mathcal{B}}
\]
Then $\mathfrak{T}$ and $\mathfrak{T}^{*}$ are Statistical Morphisms
with: 
\[
(\mathfrak{T}^{*}\circ\mathfrak{T})(P)\in\mathrm{id}_{\mathcal{A}}(P),\,\forall P\in\mathcal{P}
\]
\[
(\mathfrak{T}^{*}\circ\mathfrak{T})(Q)\in\mathrm{id}_{\mathcal{B}}(Q),\,\forall Q\in\mathcal{Q}
\]
This proves, that $\mathfrak{T}:\mathcal{P}\to\mathcal{Q}$ is a statistical
isomorphism.
\end{proof}

\section{\label{subsec:Structural-equivalence}Structural Equivalence}

The natural correspondence between a statistical model and its corresponding
canonical topological statistical model, provides a characterisation
of statistical equivalence with respect to the underlying topology.
This property allows an intuitive description of many classical concepts
of statistics, like sufficiency and completeness by homeomorphic embeddings
and surjective continuous functions. Apart of its characterising qualities
however a generic structural representation provides the opportunity
to incorporate a priori knowledge in terms of \emph{structural assumptions}.
As this knowledge usually regards the structure of the statistical
population, it has to be transferred to the statistical model.

In particular for statistical models, which are based on natural observations,
the statistic $T$ thereby does not preserve all statistical parameters,
such that also the observable structure only partially represents
the structure of the underlying statistical population. In order to
incorporate structural assumptions within the statistical model it
is therefore useful to embed the statistical model within an ``extended
statistical model'', that is able to represent the structure of the
statistical population. Then structural assumptions regarding the
statistical population correspond to constraints to the structural
embedding and therefore implicitly effect the observable structure.
Although the structural embedding theoretically does not demand the
interpretation as a statistical model, it facilitates the transfer
of concepts. In this purpose, an underlying sample space is constructed,
to extend the sample space, which is given by the statistic $T$,
by the set of unobserved statistical parameters. This provides the
notation of a partially observable measurable space.
\begin{defn*}[Partially observable measurable space]
\label{def:Partially-observable-measurable-space} \emph{Let $T\colon(\Omega,\,\mathcal{F},\,\mathrm{P})\to(S,\,\Sigma)$
be a statistic. Then a measurable space $(E,\,\mathcal{E})$ is termed
partially observable by $T$, iff there exists a measurable space
$(H,\,\mathcal{H})$, such that $(E,\,\mathcal{E})=(S\times H,\,\Sigma\otimes\mathcal{H})$.}
\end{defn*}
Due to this definition any $E$-valued random variable over $(\Omega,\,\mathcal{F},\,\mathrm{P})$
splits into a $S$-valued \emph{observable} random variable $v$ and
a $H$-valued \emph{latent} random variable $h$. Furthermore the
probability distributions over $(E,\,\mathcal{E})$ are defined by
marginal probability distributions $\mathcal{P}$ over $(S,\,\Sigma)$
and $\mathcal{Q}$ over $(H,\,\mathcal{H})$, as well as by regular
conditional probabilities $\mathrm{p}(v\mid h,\,P,\,Q)$ and $\mathrm{p}(h\mid v,\,P,\,Q)$,
where $P\in\mathcal{P}$ and $Q\in\mathcal{Q}$. This allows the definition
of an extended statistical model, which underlying sample space also
incorporates unobserved statistical parameters by latent random variables.
\begin{defn*}[Latent variable model]
\label{def:Latent-variable-model}\emph{ Let $(E,\,\mathcal{E})$
be a partially observable measurable space and $\mathcal{M}$ a set
of probability distributions over $(E,\,\mathcal{E})$. Then the tuple
$(E,\,\mathcal{E},\,\mathcal{M})$ is termed a latent variable model.}
\end{defn*}
The embedding of a given statistical model into a latent variable
statistical model that represents the statistical population provides
an intuitive description of the structural properties of the statistical
model. In particular however structural assumptions do not only effect
``\emph{known unknowns}'', as in the case of individual distributions
assumptions, but indeed ``\emph{unknown unknowns}''. This is of
particular importance for the statistical modelling of complex dynamical
systems and especially for living cells, where it has to be assumed,
that the underlying biomolecular mechanisms have only partially been
identified. In order to provide statistical inference in living cells
hence structural assumptions are required, that are able (i) to generate
tractable model spaces and (ii) to facilitate statistical inference.
Apart of the theoretical elegance of a structural embedding this approach
at first sight however seems rather impractical, since the canonical
topology is usually not tangible to an observer. Indeed without any
structural knowledge, the determination of the canonical topology
would require a complete observation of the underlying statistical
model. The situation is quite different, if the extended statistical
model may be embedded within a topological vector space $V$. Then
(i) the subspace topology of the embedding of its observable random
variables identifies the canonical topology and (ii) the algebraic
structure of the vector space facilitates the traversal of the extended
model space and therefore statistical inference. This embedding however
has to be invariant with respect to statistical equivalent statistical
models. To this end in a preceding step, a surjective Statistical
Morphism $\pi:\mathcal{\bigslant{\mathcal{M}}{\mathrm{id}_{\Sigma}}}\twoheadrightarrow\mathcal{M}_{\theta}$,
termed a \emph{statistical projection} maps the model space $\mathcal{M}$
onto a \emph{parametric family} $\mathcal{M}_{\theta}\subseteq\mathcal{M}$
and a subsequent structural embedding $e\colon\mathcal{M}_{\theta}\hookrightarrow V$
into a vector space $V$. This defines a \emph{parametrisation} of
$\mathcal{M}$.
\begin{defn*}[Parametrisation]
\label{def:Parametrisation}\emph{ Let $(S,\,\Sigma,\,\mathcal{M})$
be a statistical model, $V$ a $\mathbb{K}$-vector space, 
\[
\pi\colon\bigslant{\mathcal{M}}{\mathrm{id}_{\Sigma}}\twoheadrightarrow\mathcal{M}_{\theta}\subseteq\mathcal{M}
\]
 a statistical projection and 
\[
e\colon\mathcal{M}_{\theta}\hookrightarrow\Theta\subseteq V
\]
 an embedding. Then a mapping 
\[
\theta\colon V\to\bigslant{\mathcal{M}}{\mathrm{id}_{\Sigma}}
\]
 is termed a parametrisation of $\mathcal{M}$ over $V$, iff the
following diagram commutes:}
\begin{equation}
\xymatrix{V\ar@/^{1pc}/[rr]^{\theta} &  & \bigslant{\mathcal{M}}{\mathrm{id}_{\Sigma}}\ar@{->>}[ld]\\
 & \mathcal{M}_{\theta}\ar@{_{(}->}[lu]^{e}
}
\label{eq:def:Parametrisation}
\end{equation}
\end{defn*}
Generally a parametrisation $\theta$ is not required to be a function
in terms of unique image elements, but only with respect to the preimage of
$e\circ\pi$, such that $\mathrm{img}\theta=\bigslant{\mathcal{M}}{\mathrm{id}_{\Sigma}}$.
Therefore an arbitrary parametrisation does not provide the ability
to identify individual probability distributions by different parameters.
As the probability distributions in the parametric family $\pi(\mathcal{M})=\mathcal{M}_{\theta}\subseteq\mathcal{M}$
however are trivially observational distinguishable, the additional
claim that $\pi$ is not only a Statistical Morphism, but indeed a
statistical isomorphism provides that $\bigslant{\mathcal{M}}{\mathrm{id}_{\Sigma}}\backsimeq\bigslant{\mathcal{M}_{\theta}}{\mathrm{id}_{\Sigma}}$
and since $\bigslant{\mathcal{M}_{\theta}}{\mathrm{id}_{\Sigma}}=\mathcal{M}_{\theta}$
it follows, that $\bigslant{\mathcal{M}}{\mathrm{id}_{\Sigma}}\backsimeq\mathcal{M}_{\theta}$.
Since furthermore embeddings $e\colon\mathcal{M}_{\theta}\hookrightarrow V$
are injective it also holds that $\bigslant{\mathcal{M}}{\mathrm{id}_{\Sigma}}\backsimeq\mathrm{img}(e\circ\pi)$,
and therefore with respect to the \emph{parameter space} $\Theta\coloneqq\mathrm{dom}\theta=\mathrm{img}(e\circ\pi)$,
that $\bigslant{\mathcal{M}}{\mathrm{id}_{\Sigma}}\backsimeq\Theta$.
Then for any probability distributions $P,\,Q\in\mathcal{M}$, that
are parametrised by parameter vectors $\theta_{P},\,\theta_{Q}\in\Theta$
with $\theta_{P}\ne\theta_{Q}$ it follows that $P\notin\mathrm{id}_{\Sigma}Q$
and therefore, that $P$ and $Q$ are observational distinguishable.
Consequentially the requirement of the projection $\pi$ to be a statistical
isomorphism implicates that different parameter vectors provide different
probability distributions. Then $\theta$ is termed \emph{identifiable}.
\begin{defn*}[Identifiable parametrisation]
\label{def:Identifiable-parametrisation} \emph{Let 
\[
(S,\,\Sigma,\,\mathcal{M})\in
\]
be a statistical model, $V$ a $\mathbb{K}$-vector space and $\theta$
a parametrisation of $\mathcal{M}$ over $V$. Then $\theta$ is termed
identifiable, iff its underlying statistical projection $\pi$ is
a statistical isomorphism.}
\end{defn*}
A statistical model $(S,\,\Sigma,\,\mathcal{M})$ is termed \emph{identifiable}
if there exists an identifiable parametrisation $\theta$ over a vector
space $V$. In this case it follows, that\emph{ $\bigslant{\mathcal{M}}{\mathrm{id}_{\Sigma}}\backsimeq\bigslant{\mathcal{M}_{\theta}}{\mathrm{id}_{\Sigma}}=\mathcal{M}_{\theta}$}
and since the diagram equation \ref{eq:def:Parametrisation} is required to
commute, that $\mathcal{M}_{\theta}=\mathcal{M}$. Consequently the
notation $(S,\,\Sigma,\,\mathcal{M}_{\theta})$ naturally implicates
an identifiable parametrisation $\theta$ and therefore an identifiable
statistical model. In this case it follows that $\Theta=\mathrm{dom}\theta\backsimeq\mathrm{img}\theta\backsimeq\mathcal{M}_{\theta}$,
such that any probability distribution $P_{\theta}\in\mathcal{\ensuremath{M}_{\theta}}$
uniquely identifies a parameter vector $\theta_{P}\in\Theta$ by $\theta_{P}\coloneqq\theta^{-1}(P)$
and vice versa by $P_{\theta}\coloneqq\theta(\theta_{P})$. Then the
parameter space $\Theta$ is a \emph{parametric representation} of
the statistical model and identifiable statistical models may be compared
by their coincidence in a common parametric representation. Thereby
statistical models $(X,\,\mathcal{A},\,\mathcal{M})$ and $(Y,\,\mathcal{B},\,\mathcal{N})$
are \emph{equivalent in representation}, iff there exist identifiable
parametrisations $\theta_{\mathcal{M}}\colon\Theta_{\mathcal{M}}\to\mathcal{M}$
and $\theta_{\mathcal{N}}\colon\Theta_{\mathcal{N}}\to\mathcal{N}$
such that $\Theta_{\mathcal{M}}=\Theta_{\mathcal{N}}$. This statistically
invariant property however is completely described by the cardinality
of their identifiable parametrisations.
\begin{example*}[Cardinality of a parametrisation]
\label{exa:Cardinality-of-a-parametrisation} \emph{Let $(S,\,\Sigma,\,\mathcal{M})$
be a statistical model, $V$ a vector space, and $\theta\colon\Theta\to\mathcal{M}$
a parametrisation of $\mathcal{M}$ over $V$. Then the cardinality
of $\theta$ is given by the cardinality of the parameter space $\Theta$.}
\end{example*}
Since the cardinality of an identifiable parametrisation quantifies
the number of observational distinguishable probability distributions
within the underlying model space, statistical models are equivalent
in representation, iff they have identifiable parametrisations with
equal cardinality. Then there exist parametrisations $\theta_{\mathcal{M}}$
and $\theta_{\mathcal{N}}$, and such that $\mathcal{M}=(\theta_{\mathcal{M}}\circ\theta_{\mathcal{N}}^{-1})(\mathcal{N})$
and $\mathcal{N}=(\theta_{\mathcal{N}}\circ\theta_{\mathcal{M}}^{-1})(\mathcal{M})$.
In this case the functions $\theta_{\mathcal{NM}}=\theta_{\mathcal{M}}\circ\theta_{\mathcal{N}}^{-1}$
and $\theta_{\mathcal{M}}=\theta_{\mathcal{N}}\circ\theta_{\mathcal{M}}^{-1}$
are termed \emph{re-parametrisations}. Thereby it is important to
notice, that without further requirements re-parametrisations only
preserve the primordial distinguishability of distribution assumptions
but not their proximity structure, such that distribution assumptions,
which exclude each other in one statistical model may be nearly equivalent
in another. More for topological statistical model $(\mathcal{M},\,\tau)$
and a topological vector space $(V,\,\tau_{V})$ an identifiable parametrisation
$\theta\colon V\subseteq\Theta\to\mathcal{M}$, does not provide statistical
equivalence with respect to the subspace topology $\tau_{\Theta}$
of $\Theta$ in $V$, but only to the induced topology $\theta^{-1}(\tau)$.
Therefore re-parametrisations have to be claimed to preserve the topology,
in order to obtain statistical equivalence. As $\tau$ however is
generally not tractable by an observer it is more reasonable to require
the individual parametrisations to be continuous. This provides the
notation of a \emph{continuous parametrisation}. 
\begin{defn*}[Continuous parametrisation]
\label{def:Continuous-parametrisation} \emph{Let 
\[
(X,\,\mathcal{A},\,\mathcal{M},\,\mathcal{T})
\]
be a topological statistical model, $(V,\,\mathcal{T}_{V})$ a topological
vector space and $\theta$ a parametrisation of $\mathcal{M}$ over
$V$. Then $\theta$ is continuous, iff 
\[
\theta^{-1}\colon\mathrm{KQ}(\mathcal{M},\,\mathcal{T})\to(V,\,\mathcal{T}_{V})
\]
is a continuous embedding.}
\end{defn*}
By assuming a continuous identifiable parametrisation $\theta\colon\Theta\to\mathcal{M}$,
then the parameter space $\Theta$ naturally extends to a \emph{continuous
representation} $(\Theta,\,\mathcal{T}_{\Theta})$ of $(\mathcal{M},\,\mathcal{T})$,
where $\mathcal{T}_{\Theta}$ denotes the subspace topology of $\Theta$
in $V$. Then also 
\[
\theta\colon(\Theta,\,\mathcal{T}_{\Theta})\to\mathrm{KQ}(\mathcal{M},\,\mathcal{T})
\]
is a homeomorphism and since 
\[
\mathrm{KQ}(\Theta,\,\mathcal{T}_{\Theta})=(\Theta,\,\mathcal{T}_{\Theta})
\]
Theorem \ref{thm:3.6} proposes the existence of a statistical isomorphism
between $\Theta$ and $\mathcal{M}$, such that statistical inference
may also be derived within $(\Theta,\,\mathcal{T}_{\Theta})$ as a
subspace of $V$. Moreover any further statistical model, that has
a continuous representation, which is homeomorphic to $(\Theta,\,\mathcal{T}_{\Theta})$
is statistical equivalent to $(\mathcal{M},\,\mathcal{T})$, such
that $(\Theta,\,\mathcal{T}_{\Theta})$ is a continuous representation
of both models. In particular identifiable statistical models are
statistical equivalent, iff they are equivalent in a common continuous
representation. Therefore structure preserving re-parametrisations
naturally generalize statistical equivalence. Thereby the re-parametrisations
represent isomorphisms of a \emph{structural category} $C$, which
provides the notation of a generic \emph{structural equivalence}.
\begin{defn*}[Structural equivalence]
\label{def:Structural-equivalence}\emph{ Let $(X,\,\mathcal{A},\,\mathcal{M})$
and $(Y,\,\mathcal{B},\,\mathcal{N})$ be statistical models, $C$
a category and $(\mathcal{M},\,\varphi)$ and $(\mathcal{N},\,\vartheta)$
objects in $C$. Then $(X,\,\mathcal{A},\,\mathcal{P})$ and $(Y,\,\mathcal{B},\,\mathcal{Q})$
are structural equivalent w.r.t $C$, iff $(\mathcal{M},\,\varphi)\stackrel{C}{\backsimeq}(\mathcal{Q},\,\vartheta)$.}
\end{defn*}
As the structural equivalence of statistical models may be defined
for arbitrary underlying structural categories, it does not necessarily
correspond with statistical equivalence. In the purpose of statistical
inference however this correspondence has to be obtained by structural
assumptions. For example if the statistical models are assumed to
be identifiable, then due to theorem \ref{thm:3.6} statistical equivalence
corresponds to Kolmogorov equivalence of identifiable topological
statistical models. In this case the subset of statistical isomorphisms
that satisfy the requirement of identifiability are precisely the
homeomorphic re-parametrisations of ``identifiable topological statistical
models''. Then arbitrary additional assumptions at least have to
preserve the topology of continuous representations, such that statistical
inference may completely be derived by the evaluation of estimands
along paths within continuous representations. The great advantage
of ``structural equivalence'' over ``statistical equivalence''
however is, that it is not restricted to observable structures. This
means, that structural assumptions regarding the statistical population
may also implicitly affect the observable structure. Then statistical
inference may be obtained by a projection of an structural estimator
to the ``closest'' observable distribution assumption. However the
structural assumptions thereby have to be ``compatible'' with the
observable structure and therefore have to be statistical invariants,
as well as structural invariants of the structural category. An example
of such a \emph{structural property} of a statistical model has already
been given by the cardinality of an identifiable parametrisation,
which for the category of ``identifiable topological statistical
models'' provides the number of observational distinguishable probability
distributions. The cardinality however is unsuitable to distinguish
statistical models that comprise infinite observational distinguishable
distribution assumptions. Nevertheless in the very same manner as
the cardinality of an identifiable parametrisation determines the
number of distinguishable probability distributions, the minimal length
of structure preserving identifiable parametrisations determines the
number of distinguishable dimensions.
\begin{example*}[Length of a parametrisation]
\label{exa:Length-of-a-parametrisation} \emph{Let $(\mathcal{M},\,\mathcal{T})$
be a topological statistical model, $V$ a vector space, and $\theta\colon\Theta\to\mathcal{M}$
a parametrisation of $\mathcal{M}$ over $V$. Then the length of
$\theta$ is given by the dimension of the parameter space $\Theta$
in $V$. }
\end{example*}
The length of a given identifiable parametrisation is obviously not
a structural property of a statistical model, since it depends on
the chosen parametrisation. The minimal length of a structure preserving
identifiable parametrisation however, only depends on the structural
category, that has to be preserved. Therefore it provides a reasonable
assumption, under the precondition of an underlying category. If this
structural category is then induced by a structural assumption, then
the minimal length of a structure preserving identifiable parametrisation
is indeed a statistical property. This relationship therefore provides
a justification of \emph{minimal parametrisations} as canonical representations
of parametric families.
\begin{defn*}[Minimal parametrisation]
\label{def:Minimal-parametrisation} \emph{Let $(\mathcal{M},\,\mathcal{T})$
be a topological statistical model, $C$ a category and $\theta\colon\Theta\to\mathcal{M}$
a parametrisation of $\mathcal{M}$ over $V$. Then $\theta$ is termed
minimal w.r.t $C$ if }(i)\emph{ $\theta$ is identifiable, }(ii)\emph{
$\theta$ is an isomorphism within $C$ and }(iii) \emph{$\theta$
has a minimal length under the previous constraints.}
\end{defn*}
In the purpose to derive a minimal parametrisation, the choice of
the structural category is crucial. For example by regarding model
spaces as sets, then in the structural category $\mathbf{Set}$ the
re-parametrisations are bijections and a minimal parametrisation always
has length $1$, as long as the cardinality of the parametrisation
is smaller than that of the vector space $V$, otherwise there exists
no identifiable parametrisation over $V$. For many other structural
categories however, the derivation of a minimal parametrisation is
only hardly tractable. Then for a given statistical model and a given
vector space $V$ the question arises, if at least an identifiable
parametrisation of $\mathcal{M}$ over $V$ exists with finite length.
For example if the structural category extends the canonical topology
by a global Euclidean structure, then the structural isomorphisms
are linear mappings and the length of a minimal parametrisation equals
the minimal number of $V$-valued, linear independent random variables,
that are needed to describe any probability distribution in $\mathcal{M}$.
If $\mathcal{M}$ then contains at least one probability distribution,
which is not finite dimensional over $V$, then there exists no identifiable
parametrisation with finite length. In particular with a growing degree
of structure it is getting more difficult, or even impossible to obtain
structure preserving embeddings within vector spaces and therefore
to find appropriate identifiable parametrisations. In the purpose
to derive statistical inference by structural inference it is therefore
crucial to study the intrinsic structural properties of statistical
models.

\bibliographystyle{unsrt}
\bibliography{bibliography/articles}

\begin{thebibliography}{1}

\bibitem{Csiszar1963}
I.~Csiszár.
\newblock {Eine informationstheoretische Ungleichung und ihre Anwendung auf den
  Beweis der Ergodizitat von Markoffschen Ketten}.
\newblock {\em Magyar. Tud. Akad. Mat. Kutato Int. Kozl.}, (8):85--108, 1963.

\bibitem{Fisher1935}
Ronald~Aylmer Fisher.
\newblock {The Logic of Inductive Inference}.
\newblock {\em Journal of the Royal Statistical Society}, 98(1):39--82, 1935.

\bibitem{Kullback1959}
Solomon Kullback.
\newblock {\em {Information Theory and Statistics}}, volume~1.
\newblock Wiley, 1959.

\bibitem{Kullback1951}
Solomon Kullback and R.~A. Leibler.
\newblock {On Information and Sufficiency}.
\newblock {\em The Annals of Mathematical Statistics}, 22(1):79--86, 1951.

\bibitem{Morimoto1963}
T.~Morimoto.
\newblock {Markov processes and the H-theorem}.
\newblock {\em J. Phys. Soc. Jpn}, 18(3):328--331, 1963.

\end{thebibliography}

\end{document}